\documentclass[11pt]{article}
\usepackage[english]{babel}
\usepackage{amssymb, amscd, amsthm, amsmath,latexsym, amstext,mathrsfs,verbatim}
\usepackage{pdfpages,url}
\usepackage{comment}
\usepackage[all]{xy}
\usepackage{graphicx,enumerate}
\def\op{{\rm op}}
\def\ddA{{\rm A}}

\def\Br{{\rm Br}}
\def\SBr{{\rm SBr}}

\def\BrM{{\rm BrM}}

\def\ddB{{\rm B}}
\def\ddC{{\rm C}}
\def\ddD{{\rm D}}
\def\ddE{{\rm E}}
\def\ddI{{\rm I}}
\def\ddF{{\rm F}}
\def\ddG{{\rm G}}
\def\ddH{{\rm H}}
\def\N{{\mathbb N}}

\def\hE{{\hat E}}

\def\eps{{\epsilon}}

\def\SBrM{{\rm SBrM}}

\def\lijntje{\vrule height2.4pt depth-2pt width0.5in}

\def\vlijntje{\vrule height0.45in depth0.4pt width0.4pt}
\def\vlijn{\buildrel {\hbox to 0pt{\hss$\textstyle\circ$\hss}}\over\vlijntje}

\def\dlijntje{{\vrule height2pt depth-1.6pt
width0.5in}\llap{\vrule height4pt depth-3.6pt width0.5in}}
\def\vtriple#1\over#2\over#3{\mathrel{\mathop{\kern0pt #2}\limits_{\hbox
to 0pt{\hss$#1$\hss}}^{\hbox to 0pt{\hss$#3$\hss}}}}
\def\rvtriple#1\over#2\over#3{\mathrel{\mathop{\kern0pt #2}\limits_{\hbox
to 0pt{\hss$#3$\hss}}^{\hbox to 0pt{\hss$#1$\hss}}}}
\def\tlijntje{{\vrule height1.7pt depth-1.3pt
width0.5in}\llap{\vrule height3.0pt depth-2.6pt width0.5in}\llap{\vrule height4.3pt depth-3.9pt width0.5in}
}

\def\Dm{\vtriple{\scriptstyle n+1}\over\circ\over{}\kern-1pt\lijntje\kern-1pt
\vtriple{\scriptstyle{n}}\over\circ\over{}
\cdots\cdots\vtriple{\scriptstyle 4}\over\circ\over{}\kern-1pt\lijntje\kern-1pt
\vtriple{\scriptstyle 3}\over\circ\over{\buildrel
{\scriptstyle 2}\over\vlijn}\kern-1pt\lijntje\kern-1pt
\vtriple{1}\over\circ\over{}\kern-1pt}

\def\Dn{\vtriple{\scriptstyle n}\over\circ\over{}\kern-1pt\lijntje\kern-1pt
\vtriple{\scriptstyle{n-1}}\over\circ\over{}
\cdots\cdots\vtriple{\scriptstyle 4}\over\circ\over{}\kern-1pt\lijntje\kern-1pt
\vtriple{\scriptstyle 3}\over\circ\over{\buildrel
{\scriptstyle 2}\over\vlijn}\kern-1pt\lijntje\kern-1pt
\vtriple{1}\over\circ\over{}\kern-1pt}
\def\En{\vtriple{\scriptstyle n}\over\circ\over{}\kern-1pt\lijntje\kern-1pt
\vtriple{\scriptstyle{n-1}}\over\circ\over{}
\cdots\cdots\vtriple{\scriptstyle 5}\over\circ\over{}\kern-1pt\lijntje\kern-1pt
\vtriple{\scriptstyle 4}\over\circ\over{\buildrel
{\scriptstyle 2}\over\vlijn}\kern-1pt\lijntje\kern-1pt
\vtriple{\scriptstyle 3}\over\circ\over{}\kern-1pt\lijntje\kern-1pt
\vtriple{\scriptstyle 1}\over\circ\over{}\kern-1pt}
\def\An{\vtriple{\scriptstyle n}\over\circ\over{}\kern-1pt\lijntje\kern-1pt
\vtriple{\scriptstyle{n-1}}\over\circ\over{}\kern-1pt\lijntje\kern-1pt
\vtriple{\scriptstyle n-2}\over\circ\over{}
\cdots\cdots
\vtriple{\scriptstyle 2}\over\circ\over{}\kern-1pt\lijntje\kern-1pt
\vtriple{\scriptstyle 1}\over\circ\over{}\kern-1pt}
\def\Cn{\vtriple{\scriptstyle n-1}\over\circ\over{}
\kern-1pt\lijntje\kern-1pt\vtriple{\scriptstyle{n-2}}\over\circ\over{}
\cdots\cdots
\vtriple{\scriptstyle 2}\over\circ\over{}
\kern-1pt\lijntje\kern-1pt\vtriple{\scriptstyle 1}\over\circ\over{}
\kern-4pt{\dlijntje \kern -25pt<}\kern10pt
\vtriple{\scriptstyle 0}\over\circ\over{}\kern-1pt}
\def\Ct{\vtriple{\scriptstyle 2}\over\circ\over{}
\kern-1pt\lijntje\kern-1pt\vtriple{\scriptstyle 1}\over\circ\over{}
\kern-4pt{\dlijntje \kern -25pt<}\kern12pt
\vtriple{\scriptstyle 0}\over\circ\over{}\kern-1pt
}
\def\Bn{\vtriple{\scriptstyle n-1}\over\circ\over{}
\kern-1pt\lijntje\kern-1pt\vtriple{\scriptstyle{n-2}}\over\circ\over{}
\cdots\cdots
\vtriple{\scriptstyle 2}\over\circ\over{}
\kern-1pt\lijntje\kern-1pt\vtriple{\scriptstyle 1}\over\circ\over{}
\kern-4pt{\dlijntje \kern -25pt>}\kern10pt
\vtriple{\scriptstyle 0}\over\circ\over{}\kern-1pt}
\def\Bt{\vtriple{\scriptstyle 2}\over\circ\over{}
\kern-1pt\lijntje\kern-1pt\vtriple{\scriptstyle 1}\over\circ\over{}
\kern-4pt{\dlijntje \kern -25pt>}\kern12pt
\vtriple{\scriptstyle 0}\over\circ\over{}\kern-1pt}
\def\Es{\vtriple{\scriptstyle 6}\over\circ\over{}\kern-1pt\lijntje\kern-1pt
\vtriple{\scriptstyle 5}\over\circ\over{}\kern-1pt\lijntje\kern-1pt
\vtriple{\scriptstyle 4}\over\circ\over{\buildrel
{\scriptstyle 2}\over\vlijn}\kern-1pt\lijntje\kern-1pt
\vtriple{3}\over\circ\over{}\kern-1pt\lijntje\kern-1pt
\vtriple{\scriptstyle 1}\over\circ\over{}\kern-1pt}
\def\Ff{
\vtriple{\scriptstyle 1}\over\circ\over{}
\kern-1pt\lijntje\kern-1pt\vtriple{\scriptstyle 2}\over\circ\over{}
\kern-4pt{\dlijntje \kern -25pt<}\kern10pt
\vtriple{\scriptstyle 3}\over\circ\over{}\kern-1pt\lijntje\kern-1pt
\vtriple{\scriptstyle 4}\over\circ\over{}
\kern-1pt}
\def\Ht{
\vtriple{\scriptstyle 1}\over\circ\over{}
\kern-1pt\overset{5}{\lijntje}\kern-1pt\vtriple{\scriptstyle 2}\over\circ\over{}
\kern-1pt\lijntje\kern-1pt
\vtriple{\scriptstyle 3}\over\circ\over{}\kern-1pt}
\def\Hf{
\vtriple{\scriptstyle 1}\over\circ\over{}
\kern-1pt\overset{5}{\lijntje}\kern-1pt\vtriple{\scriptstyle 2}\over\circ\over{}
\kern-1pt\lijntje\kern-1pt
\vtriple{\scriptstyle 3}\over\circ\over{}\kern-1pt\lijntje\kern-1pt
\vtriple{\scriptstyle 4}\over\circ\over{}
\kern-1pt}
\def\In{
\vtriple{\scriptstyle 0}\over\circ\over{}
\kern-1pt\overset{n}{\lijntje}\kern-1pt\vtriple{\scriptstyle 1}\over\circ\over{}
\kern-1pt}
\def\Gt{
\vtriple{\scriptstyle 0}\over\circ\over{}
\kern-4pt{\tlijntje\kern -25pt<}\kern 10pt\vtriple{\scriptstyle 1}\over\circ\over{}
\kern-1pt}
\def\EBn{\vtriple{\scriptstyle n-1}\over\circ\over{}
\kern-1pt\lijntje\kern-1pt\vtriple{\scriptstyle{n-2}}\over\circ\over{\buildrel
{\scriptstyle -1}\over\vlijn}\cdots\cdots
\vtriple{\scriptstyle 2}\over\circ\over{}
\kern-1pt\lijntje\kern-1pt\vtriple{\scriptstyle 1}\over\circ\over{}
\kern-4pt{\dlijntje \kern -25pt<}\kern12pt
\vtriple{\scriptstyle 0}\over\circ\over{}\kern-1pt}
\def\Cn{\vtriple{\scriptstyle n-1}\over\circ\over{}
\kern-1pt\lijntje\kern-1pt\vtriple{\scriptstyle{n-2}}\over\circ\over{}
\cdots\cdots
\vtriple{\scriptstyle 2}\over\circ\over{}
\kern-1pt\lijntje\kern-1pt\vtriple{\scriptstyle 1}\over\circ\over{}
\kern-4pt{\dlijntje \kern -25pt<}\kern10pt
\vtriple{\scriptstyle 0}\over\circ\over{}\kern-1pt}
\def\ECn{\vtriple{\scriptstyle -2}\over\circ\over{}
\kern-4pt{\dlijntje \kern -25pt>}\kern12pt\vtriple{\scriptstyle n-1}\over\circ\over{}
\kern-1pt\lijntje\kern-1pt\vtriple{\scriptstyle{n-2}}\over\circ\over{}
\cdots\cdots
\vtriple{\scriptstyle 2}\over\circ\over{}
\kern-1pt\lijntje\kern-1pt\vtriple{\scriptstyle 1}\over\circ\over{}
\kern-4pt{\dlijntje \kern -25pt<}\kern12pt
\vtriple{\scriptstyle 0}\over\circ\over{}\kern-1pt}
\def\Fo{\vtriple{\scriptstyle -1}\over\circ\over{}
\kern-1pt\lijntje\kern-1pt
\vtriple{\scriptstyle 1}\over\circ\over{}
\kern-1pt\lijntje\kern-1pt\vtriple{\scriptstyle 2}\over\circ\over{}
\kern-4pt{\dlijntje \kern -25pt<}\kern12pt
\vtriple{\scriptstyle 3}\over\circ\over{}\kern-1pt\lijntje\kern-1pt
\vtriple{\scriptstyle 4}\over\circ\over{}
\kern-1pt}
\def\Ft{
\vtriple{\scriptstyle 1}\over\circ\over{}
\kern-1pt\lijntje\kern-1pt\vtriple{\scriptstyle 2}\over\circ\over{}
\kern-4pt{\dlijntje \kern -25pt<}\kern12pt
\vtriple{\scriptstyle 3}\over\circ\over{}\kern-1pt\lijntje\kern-1pt
\vtriple{\scriptstyle 4}\over\circ\over{}
\kern-1pt\lijntje\kern-1pt
\vtriple{\scriptstyle -2}\over\circ\over{}
\kern-1pt}
\def\Go{\vtriple{\scriptstyle -1}\over\circ\over{}
\kern-1pt\lijntje\kern-1pt
\vtriple{\scriptstyle 0}\over\circ\over{}
\kern-4pt{\tlijntje\kern -25pt<}\kern 12pt\vtriple{\scriptstyle 1}\over\circ\over{}
\kern-1pt}
\def\Gf{
\vtriple{\scriptstyle 0}\over\circ\over{}
\kern-4pt{\tlijntje\kern -25pt<}\kern 12pt\vtriple{\scriptstyle 1}\over\circ\over{}
\kern-1pt\lijntje\kern-1pt
\vtriple{\scriptstyle -2}\over\circ\over{}
\kern-1pt}

\newcommand{\cA}{\mathcal{A}}
\newcommand{\cB}{\mathcal{B}}

\newcommand{\R}{\mathbb R}
\newcommand{\Z}{\mathbb Z}

\newcommand{\fp}{\mathfrak{p}}

\newcommand{\fB}{\mathfrak{B}}

\newtheorem{lemma}{Lemma}[section]

\newtheorem{prop}[lemma]{Proposition}
\newtheorem{thm}[lemma]{Theorem}

\newtheorem{defn}[lemma]{Definition}

\newtheorem{rem}[lemma]{Remark}

\newtheorem{example}[lemma]{Example}

\begin{document}
\title{Brauer Algebras of multiply laced  Weyl type}
\author{Shoumin Liu
\footnotemark[1]}
\date{}
\renewcommand{\thefootnote}{\fnsymbol{footnote}}
\footnotetext[1]{Email Address: s.liu@sdu.edu.cn}
\maketitle

\begin{abstract}In this paper, we will define the Brauer algebras of Weyl types, and describe some propositions of these algebras.
Especially, we prove the result of type $\ddG_2$ to accomplish our project of  Brauer algebras of non-simply laced types.
\end{abstract}
\section{Introduction}
Brauer algebras were first studied by Brauer \cite{Brauer1937} for the Weyl Duality Theorem.
In view of these relations between the Brauer algebras and several objects of type $\ddA$, it is natural to seek analogues of these algebras for other types of spherical Coxeter group. Generalizations for other simply laced types (type $\ddD$ and $\ddE$) of the
Temperley-Lieb algebras (\cite{TL1971}) were first introduced by Graham \cite{Gra} and Fan \cite{F1995} and of Brauer algebras and
BMW algebras  by Cohen, Frenk, Gijsbers, and Wales  \cite{CFW2008}, \cite{CGW2005},  \cite{CGW2008}, \cite{CGW2009}, \cite{CW2011}. We have studied Brauer algebras of types $\ddB_n$, $\ddC_n$, $\ddF_4$ in \cite{CLY2010}, \cite{CL2011} and \cite{L2013}, respectively, and the  result
of Brauer algebra of type $\ddG_2$ will be presented here. Furthermore, the definitions and conclusions for Brauer algebras of type $\ddH_3$ and
$\ddH_4$, and $\ddI_2^n$ can be found in \cite{L20132} and \cite{L2014}, respectively. The Temperley-Lieb algebras of simply laced type which are  quotient algebras of the Hecke algebras,  can be considered as natural subalgebras of Brauer algebras (generated by those $E_i$'s in Section \ref{sect:defnsimlacety}, \cite{CW2011}).   \\
It is reasonable to ask how to define the Brauer algebras of non-simply laced types. Tits \cite{T1959} wrote about how to get  Coxeter groups
of non-simply laced type  from simply laced ones as subgroups of  invariant elements of non-trivial isomorphisms on Dynkin diagrams;
 M\"uhlherr \cite{Muehl92} showed how  to get Coxeter groups as  subgroups of Coxeter groups by admissible partitions of the canonical set of generators. Both of these two works inspired us to construct
 the Brauer algebras of non-simply laced types from simply laced ones. \\
  In this paper, our main aim is to obtain  the spherical non-simply laced Brauer algebras from simply laced types by use of the
 canonical symmetry or partitions defined on their Dynkin diagrams as in \cite{T1959} and \cite{Muehl92}( also see \cite{BC2011}, \cite{Car}, \cite{S1971}, \cite{S1998}).
 Their cellularity is also discussed. Our conclusion is summarized  in the Table \ref{mainresults}.
 In \cite{ZhiChen}, Chen defines the Brauer algebras associated to Coxeter groups by flat connections from some geometric view,
 but the author defines the product of two Temperley-Lieb generators corresponding to roots of different lengths to be zero which we do not require in our paper. In fact, some of the algebras in \cite{ZhiChen}  are  quotients of our algebras of the same types.\\
At the beginning of this paper, we first recall  necessary classical knowledge  about Coxeter groups, Coxeter diagrams, their root systems.We also introduce several  recent results about a partial order on some mutually orthogonal root
sets associated to Coxeter groups of simply laced type from \cite{CGW2006}, and also the corresponding  Brauer algebras. Subsequently we introduce the classical Brauer
algebra and prove some  results similar to those in \cite{CW2011}. In Section \ref{defn:defnnonsimply},
we introduce the definition of Brauer algebras of
Weyl  type
and  some of their basic properties, which generalizes some results in \cite{CGW2009}, \cite{CLY2010}, \cite{CL2011} and \cite{L2013}.
From Section \ref{sectionmaintheoremonG2} to Section \ref{sectiionSBrddD4}, we focus on proving the result of $\Br(\ddG_2)$ in
 Table \ref{mainresults}. In these sections, the main task is to prove that the homomorphism from $\Br(\ddG_2)$ to $\Br(\ddD_4)$ is an isomorphism,
 and the crucial step is to obtain  a basis of $\Br(\ddG_2)$.
\begin{table}
\caption{Main results}
\label{mainresults}
 \begin{center}
\begin{tabular}{|c|c|c|c|}
   \hline
   type & rank  & BSO & cellularity \\
  \hline
  $\ddC_n$(\cite{CLY2010})& $\sum_{i =0}^n\left(\sum_{p+2q = i}\frac{n!}{p! q! (n-i)!}\right)^2\,2^{n-i}\, (n-i)! $& $\ddA_{2n-1}$ & stratified \quad cellular  \\
     \hline
     $\ddB_n$(\cite{CL2011})&$2^{n+1}\cdot n!!- 2^{n}\cdot n!+(n+1)!!-(n+1)!$&$\ddD_{n+1}$&cellular\\
     \hline
     $\ddF_4$(\cite{L2013})&$14985$&$\ddE_6$&  stratified \quad cellular\\
      \hline
     $\ddH_3$(\cite{L20132})&$1045$&$\ddD_6$&  stratified \quad cellular\\
     \hline
     $\ddH_4$(\cite{L20132})&$236025$&$\ddE_8$&  stratified \quad cellular\\
      \hline
     $\ddI_2^n$(\cite{L2014}) \quad ($n\ge 5$) &$2n+n^2$($n$  odd), $2n+\frac{3}{2}n^2$($n$  even) &$\ddA_{n-1}$&  stratified \quad cellular\\
     \hline
      $\ddG_2$&$39$&$\ddD_4$& stratified \quad cellular\\
      \hline
\end{tabular}
\end{center}
\end{table}
\\In  Table \ref{mainresults}, BSO means being subalgebra of. Especially, it is proved that the  Brauer algebra of type  $\ddC_n$  is a stratified cellular algebra as defined in \cite{BO2011}.
\section{Coxeter groups and Weyl groups}
This section is based on the book \cite{B2002}.
\begin{defn}
Let $I$ be a set. A Coxeter matrix over $I$ is a matrix $M=(m_{ij})_{i,j\in I}$
where
$m_{ij}\in\N\cup \{\infty\}$ with $m_{ii}=1$ for $i\in I$, and $m_{ij}=m_{ji}>1$ for distinct $i,j\in I$.
The Coxeter group $W(M)$, or just $W$, of type $M$ is the group with presentation
$$\left<\{r_i\mid i\in I\}\mid (r_ir_j)^{m_{ij}}=1\right>.$$ This means that $W$
is freely generated by  the set $S=\{r_i\mid i\in I\}$ subject to the relations $(r_ir_j)^{m_{ij}}$ if $m_{ij}\in \N$ and no such relation
if $m_{ij}=\infty$. The pair $(W,S)$ is called a Coxeter system of type $M$.
\end{defn}
\begin{defn}
The Coxeter matrix $M$ is often described
by a labeled graph $\Gamma(M)$ whose vertex set is $I$ and in which two nodes $i$
and $j$ are joined by an edge labeled $m_{ij}$ if $m_{ij} > 2$. If $m_{ij} = 3$, then the label
$3$ of the edge $\{i,j\}$ is often omitted. If $m_{ij} = 4$, then instead of the label $4$ at
the edge $\{i,j\}$ one often draws a double bond. If $m_{ij} = 6$, then instead of the label $6$ at
the edge $\{i,j\}$ one often draws a triple bond.
This labeled graph is called
the Coxeter diagram of $M$.
\end{defn}
\begin{defn}
 When referring to a connected component  of a Coxeter matrix $M$ over $I$, we view $M$ as a labeled
graph. In other words, a connected component of $M$ is a maximal connected subset $J$ of $I$ such that
$m_{jk}=2$ for each $j\in J$ and $k\in I\setminus J$. If $M$ has a single connected component, it is called connected or irreducible.
A Coxeter group $W$ over a Coxeter diagram is called irreducible if $M$ is irreducible.
\end{defn}
\begin{defn} Let $X$ and $Y$ be free $\Z$-modules of rank $n$ with a $\Z$-bilinear form $\left<\cdot,\cdot\right>\,:\,X\times Y\rightarrow \Z$.
Let $\Phi$ be a finite subset of $X$ and suppose that for each $\alpha\in \Phi$ we have a corresponding  element $\alpha^\vee$ in $Y$.
Set $\Phi^\vee=\{\alpha^\vee\mid \alpha\in \Phi\}$. Given $\alpha\in \Phi$, we define the linear map $s_\alpha:X\rightarrow X$ by
  $$s_\alpha(x)=x-\left<x,\alpha^\vee\right>\alpha$$
  and similarly the linear map $s_\alpha^{\vee}:Y\rightarrow Y$ by
  $$s_\alpha^\vee(y)=y-\left<\alpha,y\right>\alpha^\vee.$$
\end{defn}
\begin{defn}\label{defn:rootdatum}
 Let $(X,Y,\Phi,\Phi^\vee)$ be a quadruple as above. We say   $(X,Y,\Phi,\Phi^\vee)$ is a root datum if the following three conditions
are satisfied for every root $\alpha\in \Phi$.
\begin{enumerate}[(i)]
\item $s_{\alpha}$ and $s_{\alpha}^\vee$ are reflections.
\item $\Phi$ is closed under the action of $s_{\alpha}$.
\item $\Phi^\vee$ is closed under the action of $s_{\alpha}^\vee$.
\end{enumerate}
We call elements of the finite set $\Phi$ roots and the elements of $\Phi^\vee$ coroots. The group generated by all $s_\alpha$ is called the \emph{Weyl group} of the root datum.\\
A subset $\Pi$ of $\Phi$ is called a \emph{root base} if
\begin{enumerate}[(i)]
\item $\Pi$ is a basis of $\R$-span of $\Phi$,
\item each root $\beta\in \Phi$ can be written as $\beta=\sum_{\alpha\in \Pi}k_{\alpha}\alpha$ with integer coefficients $k_\alpha$ all nonnegative or all nonpositive.
\end{enumerate}
Further the sum of all $|k_{\alpha}|$ for $\beta$ is called the height of $\beta$.
\end{defn}
\begin{defn}
 If $\Pi$ is a root base of  roots $\Phi$, then the root lengths of the simple roots are often registered in the Coxeter diagram by adding an arrow on the labeled edge between the nodes $i$ and $j$ pointing towards $j$ if the length of the root belonging to $i$ is  bigger than
the length belonging to $j$. The resulting diagram is known as the Dynkin diagram.
\end{defn}
\begin{defn} \label{Weyltype}
A Coxeter diagram  is said to be of Weyl type if it only has connected components of   type $\ddA_{n}$ ($n\ge 1$), $\ddB_{n}$ ($n\ge 2$), $\ddC_n$  ($n\ge 3$), $\ddD_n$  ($n\ge 4$), $\ddE_n$ ($6\le n\le 8$), $\ddF_4$, $\ddG_2$ as in Table \ref{DKdiagram}.
\end{defn}
It is known that the irreducible  finite (spherical) Coxeter groups are those of type $\ddA_{n}$, $\ddB_{n}$($\ddC_n$),
$\ddD_n$, $\ddE_n$ ($6\le n\le 8$), $\ddF_4$, $\ddG_2$, $\ddH_3$, $\ddH_4$, and $\ddI_2^{n}$. We  list their Coxeter  diagrams in
 Table \ref{DKdiagram}.
 \begin{table}[!htb]
\caption{Coxeter diagrams of spherical types}\label{DKdiagram}
\begin{center}
\begin{tabular}{c|c}
name&diagram\\
\hline
$\ddA_n$&$\An$\\
$\ddD_n$&$\Dn$\\
$\ddE_n$, $6\le n\le 8$&$\En$\\
$\ddB_n$&$\Bn$\\
$\ddC_n$&$\Cn$\\
$\ddF_4$&$\Ff$\\
$\ddH_3$&$\Ht$\\
$\ddH_4$&$\Hf$\\
$\ddI_2^n$&$\In$\\
$\ddG_2$&$\Gt$
\end{tabular}
\end{center}
\end{table}
 Let $V=\R^{I}$ be the Euclidean space of dimension $l=|I|$ with basis  $\{\alpha_i\}_{i\in \Pi}$. Let $B$ be the bilinear form over
  $V$ such that $$B(\alpha_i,\alpha_j)=-2\cos\frac{\pi}{m_{ij}}.$$ Let $\rho_i(x)=x-B(x,\alpha_i)\alpha_i$,
  for $x\in V$, $i\in I$.
  Then the map $\rho:I\rightarrow {\rm GL}(V)$, $\rho(r_i)=\rho_i$ can determine a faithful representation of
  $W(M)$ on  $\R^{I}$ (\cite{B2002} or \cite{BC2011}).
\begin{defn}
The subset $\Phi=\cup_{i\in I}\rho(W)\alpha_i$ of $V$ is called the root system of $W$. The subset
of $\Phi^+\subset \Phi$ ($\Phi^-\subset \Phi$ ) of all elements with nonnegative (nonpositive) coefficients for the basis $\{\alpha_i\}_{i\in I}$ is called the set of positive roots (negative roots) of $W$. Each root $\alpha_i$ is called a simple root of $W$ for $i\in I$.
\end{defn}
It is known from \cite[Theorem 3, Chapter 6]{B2002} that
\begin{eqnarray*}
\Phi=\Phi^+\cup\Phi^-,\quad \Phi^+=-\Phi^-,\quad \Phi^+\cap\Phi^-=\emptyset.
\end{eqnarray*}
\section{A poset of simply laced type}\label{sect:poset}
Let $Q$ be a spherical Coxeter diagram of simply laced type, i.e., its connected components are of type
$\ddA$, $\ddD$, $\ddE$ as listed in Table \ref{DKdiagram}. This section is to summarize some results in \cite{CGW2006}.

When $Q$ is  $\ddA_n$, $\ddD_n$, $\ddE_6$, $\ddE_7$, or $\ddE_8$, we denote it as $Q\in{\rm ADE}$.
Let $(W, T)$ be the Coxeter system of type $Q$ with $T=\{R_1,\ldots,R_n\}$ associated to the diagram of
$Q$ in Table \ref{DKdiagram}.
Let $\Phi$ be the root system of type $Q$,  let $\Phi^+$ be its positive root system, and let $\alpha_i$ be the simple root
associated to the node $i$ of $Q$. We are interested in sets $B$ of mutually commuting reflections, which has a bijective correspondence with sets of
mutually orthogonal roots of $\Phi^+$, since each reflection in $W$ is uniquely determined by a positive root and vice versa.
\begin{rem}\label{rem:positiveaction}
The action of $w\in W$ on $B$ is given by conjugation in case $B$ is described by reflections and given by
$w\{\beta_1,\ldots, \beta_p\}=\Phi^+\cap \{\pm w\beta_1,\ldots, \pm w\beta_p \}$, in case $B$ is described by positive roots.
For example, $R_4R_1R_2R_1\{\alpha_1+\alpha_2, \alpha_4\}=\{\alpha_1+\alpha_2,\alpha_4\}$, where $Q=\ddA_4$.
\end{rem}
For $\alpha$, $\beta\in \Phi$, we write $\alpha\sim\beta$ to denote $|(\alpha,\beta)|=1$. Thus, for $i$ and $j$ nodes of
$Q$, we have $\alpha_i\sim\alpha_j$ if and only if $i\sim j$.
\begin{defn}
Let $\mathfrak{B}$ be a   $W$-orbit of sets of  mutually orthogonal positive roots. We say that
$\mathfrak{B}$ is an \emph{admissible orbit} if for each $B\in \mathfrak{B}$, and $i$, $j\in Q$ with $i\not\sim j$
 and $\gamma$, $\gamma-\alpha_i+\alpha_j\in B$ we have $r_iB=r_jB$, and each element in $\mathfrak{B}$ is called an admissible root set.
\end{defn}
This is the definition from \cite{CGW2006},  and there is another equivalent definition in \cite{CFW2008}. We also state it here.
\begin{defn}
 Let $B\subset\Phi^+$ be a mutually orthogonal root set. If for all $\gamma_1$, $\gamma_2$, $\gamma_3\in B$
 and $\gamma\in \Phi^+$, with $(\gamma,\gamma_i)=1$, for
 $i=1$, $2$, $3$, we have $2\gamma+\gamma_1+\gamma_2+\gamma_3\in B$, then $B$ is called an admissible root set.
 \end{defn}
 By these two definitions, it follows that the intersection of two admissible root sets are admissible.
 It can be checked by definition that the intersection of two admissible sets are still admissible.  Hence
for a given  set $X$ of mutually orthogonal positive roots, the unique smallest admissible set containing
$X$ is called the admissible closure of $X$, and denoted as $X^{\rm cl}$ (or $\overline{X}$). Up to
the action of the corresponding Weyl groups,
all admissible root sets of type $\ddA_n$, $\ddD_n$,  $\ddE_6$, $\ddE_7$, $\ddE_8$ have appeared in  \cite{CFW2008}, \cite{CGW2009} and \cite{CW2011},
and are listed
in  Table \ref{table:admADE}. In the table, the
set
$Y(t)^*$  consists of all $\alpha^*$ for  $\alpha\in Y(t)$, where $\alpha^*$ is    the unique
positive root orthogonal to $\alpha$ and all other positive roots orthogonal
to $\alpha$ for type $\ddD_n$ with $n> 4$.  For type $\ddD_n$, if we considier the root systems are realized in $\R^{n}$,
with $\alpha_1=\epsilon_2-\epsilon_1$, $\alpha_2=\epsilon_2+\epsilon_1$, $\alpha_i=\epsilon_i-\epsilon_{i-1}$, for $3\le i\le n$,
then $\Phi^+=\{\epsilon_j\pm\epsilon_i\}_{1\le i<j\le n}$, then  $(\epsilon_j\pm\epsilon_i)^*=\epsilon_j\mp\epsilon_i$.  For $\ddD_4$,
the $t$ can be $0$, $1$, $2$, $3$, which means the number of nods in the  coclique.
When $t=2$, although in the Dynkin diagram $\{\alpha_1,\alpha_2\}$ and $\{\alpha_1,\alpha_4\}$ are symmetric,  they are  in the
 different orbits under the  Weyl group's actions.  Then the admissible root sets for $\ddD_4$ can be written as the $W(\ddD_4)$'s
 orbits of $\emptyset$, $\{\alpha_3\}$,  $\{\alpha_1,\alpha_2\}$,  $\{\alpha_1,\alpha_4\}$, and $\{\alpha_1,\alpha_2,\alpha_4,\alpha_1+\alpha_2+\alpha_4+2\alpha_3\}.$
\begin{table}
\caption{Admissible root sets of simply laced type}
\label{table:admADE}
 \begin{center}
\begin{tabular}{|c|c|}
\hline
$Q$&representatives of orbits \, under\, $W(Q)$\\
\hline
$\ddA_n$&$\{\alpha_{2i-1}\}_{i=1}^{t}$,\,$0\le t\le \left\lfloor{(n+1)/2}\right\rfloor. $\\
\hline
$\ddD_n$&$Y(t)=\{\alpha_{n+2-2i}, \alpha_{n-2},\ldots, \alpha_{n+2-2t}\}$ \, $0\le t\le \left\lfloor{n/2}\right\rfloor. $\\
        & $\{\alpha_{n+2-2i}, \alpha_{n-2},\ldots, \alpha_{4}, \alpha_1\}$  \text{if}    $2|n$\\
        & $Y(t)\cup Y(t)^*$ \, $0\le t\le \left\lfloor{n/2}\right\rfloor$\\
\hline
$\ddE_6$ & $\emptyset$, $\{\alpha_6\}$, $\{\alpha_6, \alpha_4\}$,  $\{\alpha_6, \alpha_2, \alpha_3\}^{\rm cl}$ \\
\hline
$\ddE_7$ & $\emptyset$, $\{\alpha_7\}$, $\{\alpha_7, \alpha_5\}$,  $\{\alpha_5, \alpha_5, \alpha_2\}$, $\{\alpha_7, \alpha_2, \alpha_3\}^{\rm cl}$,
                               $\{\alpha_7, \alpha_5, \alpha_2, \alpha_3\}^{\rm cl}$   \\
\hline
$\ddE_8$ & $\emptyset$, $\{\alpha_8\}$, $\{\alpha_8, \alpha_6\}$,   $\{\alpha_8, \alpha_2, \alpha_3\}^{\rm cl}$,
                               $\{\alpha_8, \alpha_5, \alpha_2, \alpha_3\}^{\rm cl}$   \\
\hline
\end{tabular}
\end{center}
\end{table}

\begin{example} If $Q=\ddD_4$, the root set $\{\alpha_1, \alpha_2, \alpha_4\}$ is mutually orthogonal but not admissible,
and its admissible closure is $\{\alpha_1, \alpha_2, \alpha_4, \alpha_1+\alpha_2+2\alpha_3+\alpha_4\}$.
\end{example}
\begin{defn}
\label{df:cA}
Let $\cA$ denote the collection of all admissible subsets of $\Phi$ consisting of
mutually orthogonal  positive roots.
Members of $\cA$ are called admissible sets.
\end{defn}
Now we consider the actions of  $R_i$ on an admissible $W$-orbit $\fB$. When $R_iB\neq B$, We say that
$R_i$ lowers $B$ if there is a root $\beta\in B$ of minimal height  among those moved by $R_i$ that
satisfies $\beta-\alpha_i\in \Phi^+$ or $R_iB<B$.
We say that $R_i$ raises $B$ if there is a root $\beta\in B$ of minimal height among
those moved by $R_i$ that satisfies $\beta+\alpha_i\in \Phi^+$ or $R_iB>B$. By this
we can set an  partial order on $\fB=WB$. The poset $(\fB, <)$ with this minimal ordering is called the monoidal poset (with respect to $W$) on $\fB$ (so $\fB$ should be admissible for the poset to be monoidal). If $\fB$ just consists of sets
of a single root, the order is determined by the canonical height function on roots. There is an important conclusion in \cite{CGW2006},  stated  below.
This theorem plays a crucial role in obtaining a basis for Brauer algebra of simply laced type in \cite{CFW2008}.
\begin{thm}\label{thm:maximal}
 There is a unique maximal element in $\fB$.
 \end{thm}
 \section{Brauer algebras of simply laced types}\label{sect:defnsimlacety}
The Brauer algebra of type $\ddA$ was first introduced in \cite{Brauer1937} for studying the invariant theory of
orthogonal group.
In \cite{CFW2008}, it is extended to  simply laced types, in the way described below.
\begin{defn}\label{1.1}\label{c3.1.1}
Let $R$ be a commutative ring with invertible element $\delta$ and let $Q$
be a simply laced Coxeter diagram. The \emph{Brauer algebra of type $Q$ over
$R$ with loop parameter $\delta$}, denoted $\Br(Q,R,\delta)$, is the $R$-algebra generated by $R_i$ and $E_i$, for each node $i$ of $Q$
subject to the following relations, where $\sim$ denotes
adjacency between nodes of $Q$.
\begin{eqnarray}
R_{i}^{2}&=&1          \label{1.1.2} \\
E_{i}^{2}&=&\delta E_{i}   \label{1.1.4} \\
R_iE_i&=&E_iR_i \,=\, E_i     \label{1.1.3} \\
R_iR_j&=&R_jR_i, \,\, \mbox{for}\, \it{i\nsim j} \label{1.1.5} \\
E_iR_j&=&R_jE_i,\,\, \mbox{for}\, \it{i\nsim j}  \label{1.1.6} \\
E_iE_j&=&E_jE_i,\,\, \mbox{for}\, \it{i\nsim j}    \label{1.1.7} \\
R_iR_jR_i&=&R_jR_iR_j, \,\, \mbox{for}\, \it{i\sim j}  \label{1.1.8} \\
R_jR_iE_j&=&E_iE_j ,\,\, \mbox{for}\, \it{i\sim j}       \label{1.1.9} \\
R_iE_jR_i&=&R_jE_iR_j ,\,\, \mbox{for}\, \it{i\sim j}     \label{1.1.10}
\end{eqnarray}
As before, we call $\Br(Q) := \Br(Q,\Z[\delta^{\pm1}],\delta)$
the Brauer algebra of type $Q$
and denote by $\BrM(Q)$
the submonoid of the multiplicative monoid of $\Br(Q)$ generated by $\delta^{\pm 1}$ and
all $R_i$ and $E_i$.
 \end{defn}
 \begin{rem}\label{morerel}
As a consequence of the  relations of Brauer algebras  $\Br(Q)$ of simply laced type, it is straightforward to show that
the following relations hold in $\Br(Q)$ for all nodes $i$, $j$, $k$ with
$i\sim j \sim k$ and $i\not\sim k$ (see \cite[Lemma 3.1]{CFW2008}).
\begin{eqnarray}
E_iR_jR_j&=&E_iE_j \label{3.1.1} \\ R_jE_iE_j &=& R_i E_j \label{3.1.2} \\
E_iR_jE_i &=& E_i \label{3.1.3} \\ E_jE_iR_j &=& E_j R_i \label{3.1.4} \\
E_iE_jE_i &=& E_i \label{3.1.5} \\ E_jE_iR_k E_j &=& E_jR_iE_k E_j
\label{3.1.6} \\ E_jR_iR_k E_j &=& E_jE_iE_k E_j \label{3.1.7}
\end{eqnarray}
\end{rem}
 For any $\beta\in\Phi^+$ and $i\in\{1,\ldots,n\}$, there exists a $w\in
W$ such that $\beta = w\alpha_i$. Then $R_\beta := wR_iw^{-1}$ and
$E_\beta := wE_iw^{-1}$ are well defined (this is well known from Coxeter
group theory for $R_\beta$;
see \cite[Lemma 4.2]{CFW2008} for $E_\beta$).  If
$\beta,\gamma\in\Phi^+$ are mutually orthogonal, then $E_\beta$ and
$E_\gamma$ commute (see \cite[Lemma 4.3]{CFW2008}). Hence, for $B\in\cA$, we
 define the product
\begin{eqnarray}
\label{eqn:EprodB}
E_B &=& \prod_{\beta\in B} E_\beta,
\end{eqnarray}
which is a quasi-idempotent, and the normalized version
\begin{eqnarray}
\label{eqn:EhatB}
\hE_B &=& \delta^{-|B|} E_B,
\end{eqnarray}
which is an idempotent element of the Brauer monoid.
For a mutually orthogonal root subset $X\subset \Phi^+$, we have
\begin{eqnarray}
\label{eqn:Ecloure}
E_{X^{\rm cl}}=\delta^{|X^{\rm cl}\setminus X|}E_{X}.
\end{eqnarray}
%
Let $C_X=\{i\in Q\mid \alpha_i \perp X\}$ and
let $W(C_{X})$ be the subgroup generated by the
generators of nodes in $C_X$.
The subgroup $W(C_{X})$ is called the \emph{centralizer} of $X$.
The normalizer of $X$, denoted by $N_{X}$ can be defined as
$$N_{X} =\{w\in W\mid E_X w=w E_X\}.$$
We let $D_X$ denote a set of right coset representatives for $N_X$ in $W$.\\
In \cite[Definition 3.2]{CFW2008}, an action of the Brauer monoid
$\BrM(Q)$ on the collection $\cA$ of admissible root sets
in $\Phi^+$ was indicated below, where  $Q\in {\rm ADE}$.
\begin{defn}\label{eq:aboveaction}
There is an action of the Brauer monoid $\BrM(Q)$ on the collection
$\cA$.  The generators $R_i$ $(i=1,\ldots,n)$ act by the natural action of
Coxeter group elements on its  positive root sets as in Remark \ref{rem:positiveaction},
and the element $\delta$ acts as the identity,
and the action of $E_i$ $(i=1,\ldots,n)$ is defined by
\begin{equation}
E_i B :=\begin{cases}
B & \text{if}\ \alpha_i\in B, \\
(B\cup \{\alpha_{i}\})^{\rm cl} & \text{if}\ \alpha_i\perp B,\\
R_\beta R_i B & \text{if}\ \beta\in B\setminus \alpha_{i}^{\perp}.
\end{cases}
\end{equation}
\end{defn}

We will refer to this action as the admissible set action.
This monoid action plays an important role in getting a basis of $\BrM(Q)$ in \cite{CFW2008}. For the basis, we state one  conclusion from \cite[Proposition 4.9]{CFW2008} below.
 \begin{prop}\label{ADErewform}
  Each element of the Brauer monoid $\BrM(Q)$ can be written in the form $$\delta^k uE_{X}zv,$$ where
 $X$ is the highest element from one $W$-orbit in $\cA$, $u$, $v^{-1}\in D_X$, $z\in W(C_X)$, and $k\in \Z$.
 \end{prop}
 There is a  more general version for simply laced types in \cite{CW2011}.
We keep  notation as  in \cite[Section 2]{CW2011} and first introduce some basic concepts.
Let $Q$ be  the  diagram  of a  connected finite simply laced Coxeter
group (type $\ddA_n$, $\ddD_n$, $\ddE_6$,  $\ddE_7$,  $\ddE_8$).
Then $\BrM(Q)$ is the associated Brauer monoid as  in Definition \ref{1.1}.
  By $B_Y$ we denote
the admissible closure  of $\{\alpha_i|i\in Y\}$, where $Y$ is a coclique
of $Q$. The set $B_Y$ is a minimal element in the $W(Q)$-orbit of $B_Y$ which is endowed with a  poset  structure
induced by the partial ordering $<$
defined on $W(Q)$-orbits  in $\cA$ in Section \ref{sect:poset}. If $d$ is the Hasse diagram distance for $W(Q)B_Y$
from $B_Y$ to the unique maximal element, then for $B\in W(Q)B_Y$ the height of $B$, already used in Definition
notation $\rm{ht}$$(B)$, is $d-l$, where $l$ is the distance in the
Hasse diagram from $B$ to the maximal element. The Figure \ref{fig:Hassediagram} is a Hasse diagram of admissible sets of type $\ddA_4$ with $2$  mutually orthogonal positive roots.
 As indicated in Theorem \ref{thm:maximal}, the set $\{\alpha_1+\alpha_2+\alpha_3, \alpha_2+\alpha_3+\alpha_4\}$ is the maximal root set in its $W(\ddA_4)$-orbit.
\begin{figure}[!htb]
\begin{center}
\includegraphics[width=.9\textwidth,height=.5\textheight]{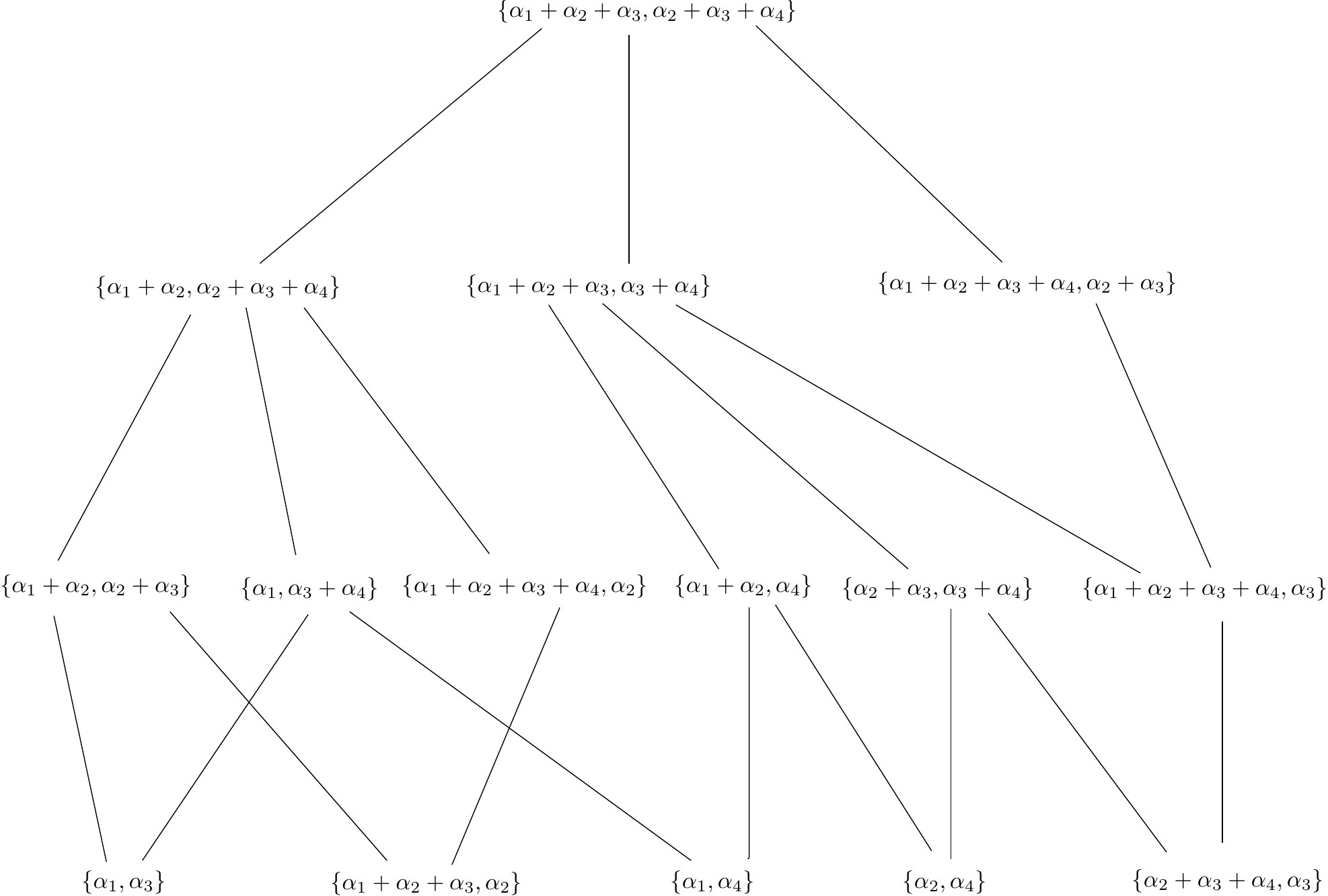}
\end{center}
\caption{A Hasse diagram of type $\ddA_4$.}
\label{fig:Hassediagram}
\end{figure}
\begin{thm} \label{thm:genralwriting}
(\cite[Theorem 2.7]{CW2011}) Each monomial  $a$ in $\BrM(Q)$ can be
uniquely written as $\delta^{i} a_{B} \hat{E}_Y h a_{B'}^{\rm op}$ for some $i\in \Z$ and $h\in W(Q_{Y})$,
where $W(Q_{Y})$ is the group of invertible elements in $\hat{E_Y}W(Q)\hat{E_Y}$,
$B=a\emptyset$, $B^{'}=\emptyset a $, $a_{B}\in \BrM(Q)$, $a_{B'}^{\rm op}\in \BrM(Q)$ and
\\ (i) $a\emptyset=a_{B}\emptyset=a_{B}B_Y$,  $\emptyset a= \emptyset a_{B'}^{\rm op}= B_Y a_{B'}^{\rm op}$,
\\(ii) $\rm{ht}$$(B)=$\rm{ht}$(a_{B})$, $\rm{ht}$$(B')=$\rm{ht}$(a_{B'}^{\rm op})$.
\end{thm}
 \section{Definition of Brauer algebras of type $\ddB_n$, $\ddC_n$, $\ddF_4$, $\ddG_2$}\label{defn:defnnonsimply}
We denote $M\in {\rm BCFG}$, if $M$ is a Dynkin diagram of type $\ddB_n$, $\ddC_n$, $\ddF_4$ or $\ddG_2$.
We abuse the notation $i\in M$ for a node  $i$ of $M$.
\begin{defn} \label{0.1}\label{c3.0.1}\label{c4.0.1}\label{c7.0.1}
Let $R$ be a commutative ring with invertible element
$\delta$ and $M$ be a Dynkin diagram of Weyl type.  For $n\in \N$, the \emph{Brauer algebra of type $M$} over $R$
with loop parameter $\delta$, denoted by $\Br(M,R,\delta)$, is the
$R$-algebra generated by $\{r_i,e_i\}_{i\in M}$ subject to the following relations.
For each $i\in M$,
\begin{eqnarray}
r_{i}^{2}&=&1,    \label{0.1.3} \label{c3.0.1.3}\label{c4.0.1.3} \label{c7.0.1.3}\label{c6.0.1.3}
\\
r_ie_i &= & e_ir_i \,=\, e_i,  \label{0.1.4} \label{c3.0.1.4} \label{c4.0.1.4}\label{c7.0.1.4}\label{c6.0.1.4}
\\
e_{i}^{2}&=&\delta^{\kappa_i}e_{i};\label{0.1.5}\label{0.1.6}\label{c6.0.1.5}\label{c3.0.1.5}\label{c3.0.1.6}\label{c4.0.1.5}\label{c4.0.1.6}\label{c7.0.1.5}\label{c7.0.1.6}
\end{eqnarray}
for $i$, $j\in M$  not adjacent to each other, namely $\vtriple{\scriptstyle i}\over\circ\over{}\kern-1pt\quad\kern-1pt
\vtriple{\scriptstyle{j}}\over\circ\over{}\kern-1pt$ ,
\begin{eqnarray}
r_ir_j&=&r_jr_i,   \label{0.1.7}\label{c3.0.1.7}\label{c4.0.1.7}\label{c6.0.1.6}
\\
e_ir_j&=&r_je_i,     \label{0.1.8}\label{c3.0.1.8}\label{c4.0.1.8}\label{c6.0.1.7}
\\
e_ie_j&=&e_je_i;    \label{0.1.9} \label{c3.0.1.9}\label{c4.0.1.9}  \label{c6.0.1.8}
\end{eqnarray}
for $i$, $j\in M$  and $\vtriple{\scriptstyle i}\over\circ\over{}\kern-1pt\lijntje\kern-1pt
\vtriple{\scriptstyle{j}}\over\circ\over{}\kern-1pt$ ,
\begin{eqnarray}
r_ir_jr_i&=&r_jr_ir_j,  \label{0.1.10}\label{c3.0.1.10} \label{c4.0.1.10} \label{c6.0.1.9}
\\
r_jr_ie_j&=&e_ie_j ,             \label{0.1.13}\label{c3.0.1.13} \label{c4.0.1.11}\label{c6.0.1.10}
\\
r_ie_jr_i&=&r_je_ir_j;         \label{0.1.15} \label{c3.0.1.15} \label{c4.0.1.12}\label{c6.0.1.11}
\end{eqnarray}
for $i$, $j\in M$ and $\vtriple{\scriptstyle i}\over\circ\over{}
\kern-4pt{\dlijntje \kern -25pt>}\kern12pt
\vtriple{\scriptstyle j}\over\circ\over{}\kern-1pt$ ,
\begin{eqnarray}
r_{j}r_ir_{j}r_{i}&=&r_ir_{j}r_ir_{j},             \label{0.1.11} \label{c3.0.1.11}\label{c4.0.1.13}
 \\
r_{j}r_ie_{j}&=&r_ie_{j},                               \label{0.1.14}\label{c3.0.1.14} \label{c4.0.1.14}
\\
  r_{j}e_ir_{j}e_i&=&e_ie_{j}e_i,                                                        \label{0.1.19} \label{c3.0.1.19}\label{c4.0.1.15}
\\
(r_{j}r_ir_{j})e_i&=&e_i(r_{j}r_ir_{j}),                                                            \label{0.1.20}\label{c3.0.1.20}\label{c4.0.1.16}
\\
e_{j}r_{i}e_{j}&=&\delta e_{j},                                                \label{0.1.12}\label{c3.0.1.12} \label{c4.0.1.17}
\\
e_{j}e_ie_{j}&=&\delta e_{j}   ,                                                 \label{0.1.16}\label{c3.0.1.16}\label{c4.0.1.18}
\\
e_{j}r_i r_{j}&=&e_{j}r_{i}     ,                                                     \label{0.1.17} \label{c3.0.1.17} \label{c4.0.1.19}
\\
 e_{j}e_ir_{j}&=&e_{j}e_i        ;                                                    \label{0.1.18} \label{c3.0.1.18}\label{c4.0.1.20}
\end{eqnarray}
for $i$, $j\in M$ and $\vtriple{\scriptstyle i}\over\circ\over{}
\kern-4pt{\tlijntje\kern -25pt<}\kern 12pt\vtriple{\scriptstyle j}\over\circ\over{}
\kern-1pt$ ,
\begin{eqnarray}
r_ie_je_i&=&r_je_i,   \label{c7.0.1.7}
\\
 e_ie_jr_i&=&e_ir_j,     \label{c7.0.1.8}
 \\
e_jr_ie_jr_ie_j&=&e_j,                                                \label{c7.0.1.12}
\\
      e_jr_ie_jr_ir_j&=&e_jr_ir_jr_i,         \label{c7.0.1.13}
\\
e_ir_je_i&=&\delta^2 e_i                                \label{c7.0.1.14}
\\
r_jr_ie_jr_ie_j&=&r_ir_jr_ie_j. \label{c7.0.1.16}
\\
(r_jr_i)^6&=&1.    \label{c7.0.1.17}
\end{eqnarray}
The parameter $\kappa_i\in \N$ is given  below, \\
for type $\ddC_n$, $\kappa_0=1$, $\kappa_i=2$ for $1\le i\le n-1$;\\
for type $\ddB_n$, $\kappa_0=2$, $\kappa_i=1$ for $1\le i\le n-1$;\\
for type $\ddF_4$, $\kappa_1=\kappa_2=2$, $\kappa_3=\kappa_4=1$;\\
for type $\ddG_2$, $\kappa_0=3$, $\kappa_1=1$.\\
If $R = \Z[\delta^{\pm1}]$ we write
$\Br(M)$ instead of $\Br(M,R,\delta)$ and speak of  the
Brauer algebra of type $M$.  The submonoid of the multiplicative
monoid of $\Br(M)$ generated by $\delta$, $\delta^{-1}$ and $\{r_i,e_i\}_{i\in M}$
 is denoted by
$\BrM(M)$. It is the monoid of monomials in $\Br(M)$ and will be
called the Brauer monoid of type $M$.
\end{defn}
Recall the Lemma 4.1 from \cite{CLY2010} in  the following.
\begin{lemma}\label{furtherrel}
 If $i$, $j\in M$ and $\vtriple{\scriptstyle i}\over\circ\over{}
\kern-4pt{\dlijntje \kern -25pt>}\kern12pt
\vtriple{\scriptstyle j}\over\circ\over{}\kern-1pt$ , the following
equations hold.
\begin{eqnarray}
r_{j}e_i e_{j}&=&e_{i}e_{j}     \label{4.1.2}
\\
e_ie_{j}e_{i}&=&e_ir_{j}e_i  \label{4.1.1}
\\
e_{j}r_{i}r_{j}e_{i}&=&e_{j}e_{i}       \label{4.1.3}
\\
r_ir_{j}e_{i}r_{j}&=&r_{j}e_{i}r_{j}r_i       \label{4.1.4}
\\
e_ir_{j}e_ir_{j}&=&e_ie_{j}e_i                \label{4.1.5}
\end{eqnarray}
\end{lemma}
For the triple bond, we have some results in Lemma \ref{c7.2}.
%
%

Similar to the case for Brauer algebras of simply laced types,
there is a natural anti-involution on $\Br(M,R,\delta)$  linearly induced by
$$x_1 x_2\ldots x_n\mapsto x_n\ldots x_2 x_1$$
with each $x_i$ being the generator of $\Br(M,R,\delta)$.

\begin{prop} \label{prop:opp}\label{c3.prop:opp}\label{c4.prop:opp}\label{c7.rem:opp}\label{c7.prop:opp}
The identity map on
$\{\delta, r_i, e_i \mid i = 0,\ldots,n-1\}$ extends to a unique
anti-involution on the Brauer algebra $\Br(M,R,\delta)$.
\end{prop}

\begin{proof}
It suffices to check the defining relations given in Definition \ref{1.1}
still hold under the anti-involution.  An easy inspection shows that all
relations involved in the definition are invariant under $\op$, except for
(\ref{0.1.13}), (\ref{0.1.14}), (\ref{0.1.19}), (\ref{0.1.17}),
(\ref{0.1.18}), (\ref{c7.0.1.7}), (\ref{c7.0.1.8}), (\ref{c7.0.1.13}), and (\ref{c7.0.1.16}).  The relation obtained by applying $\op$ to (\ref{0.1.13})
holds as can be seen by using (\ref{0.1.15}) followed by (\ref{0.1.3})
together with (\ref{0.1.13}).  The equality (\ref{0.1.17}) is the op-dual of
(\ref{0.1.14}). Finally, (\ref{4.1.2}), (\ref{4.1.5}), (\ref{c7.0.1.7}), and  (\ref{c7.0.1.13}) state  the $\op$-dual of (\ref{0.1.18}), (\ref{0.1.19}),  (\ref{c7.0.1.8}), and  (\ref{c7.0.1.16}), respectively. Hence our claim holds.
\end{proof}
This anti-involution is denoted by the superscript $\op$, so the map is given by $x\mapsto x^{\op}$.\\
Let $M\in {\rm BCFG}$,  $Q\in {\rm ADE}$ be the corresponding type in the first column and the third
 column in Table \ref{mainresults}, $\sigma$ be the nontrivial diagram automorphism
to obtain $M$ in \cite{T1959}, which is
 \begin{enumerate}[(i)]
 \item $\sigma=\prod_{i=1}^{n-1}(i,2n-i)$ for $\ddC_n$ in $\ddA_{2n-1}$,
 \item $\sigma=(1,2)$ for $\ddB_n$ in $\ddD_{n+1}$,
 \item $\sigma=(1,6)(3,5)$ for $\ddF_4$ in $\ddE_6$,
 \item $\sigma=(1,2,4)$ for $\ddG_2$ in $\ddD_4$,
 \end{enumerate}
 and $\cA_\sigma$ be the subset of $\sigma$-invariant admissible root sets of $\cA$ under $\sigma$. \\It is well known that the image of $\{\alpha_i\mid i\in Q\}$ under the  map
 $\alpha_i\mapsto \sum_{t=1}^{|\sigma|}\sigma^t(\alpha_i)/|\sigma|$, namely
$$\left\{\sum_{t=1}^{|\sigma|}\sigma^t(\alpha_i)/|\sigma|\quad\mid i\in Q\right\}$$ consists of  the simple roots of type $M$, where $|\sigma|$ is the order of
$\sigma$.
We extend the map  to a linear map $\fp: \R^{|Q|}\rightarrow \R^{|Q|}_\sigma= \R^{|M|}$,
$x\mapsto \sum_{t=1}^{|\sigma|}\sigma^t(x)/|\sigma|$.  Let $\Phi^+$ be the positive roots of type $Q$
with respect to $\{\alpha_i\}_{i\in Q}$ and $\Psi^+$ be the positive roots of type $M$ with respect to   $\{\sum_{t=1}^{|\sigma|}\sigma^t(\alpha_i)/\mid\sigma|\mid i\in Q \}$.
\\We next consider particular sets of mutually orthogonal positive roots in
$\Psi^+$, and relate them to symmetric admissible sets in $\cA$.

\begin{defn}\label{df:admissible}\index{\emph{admissible}}
Denote by $\cB'$ the collection of all sets of mutually orthogonal roots in
$\Psi^+$ and by $\cA_\sigma$ the subset of $\sigma$-invariant elements of
$\cA$.  As $\fp$ sends positive roots of $\Phi$ to positive roots of $\Psi$,
it induces a map $\fp : \cA_\sigma\to\cB'$ given by $\fp(B) =
\left\{\fp(\alpha)\mid \alpha\in B\right\}$ for $B\in \cA_\sigma$.  An
element of $\cB'$ will be called  admissible if it lies in the image of an admissible set
under
$\fp$. The set of all admissible elements of $\cB'$ will be denoted $\cB$.
\end{defn}
\begin{defn}\index{\emph{admissible closure}}
Suppose that $X\subset \Psi^+$ is a mutually orthogonal root set. If $X$ is
a subset of some admissible root set, then the minimal admissible set
containing $X$ is called the \emph{admissible closure} of $X$, denoted by
$\overline{X}$ or $X^{\rm cl}$.
\end{defn}
Some examples can be seen in \cite[Remark 5.4]{CLY2010}.

Recall for any $\beta\in \Psi^+$, the height of $\beta$ is the sum of coefficients of the simple roots if we write
$\beta$ as the linear combination of simple roots. For  the root
systems of type $\rm{BCFG}$, under the action of corresponding Weyl group, there are two orbits in $\Psi$, distinguished  by
the Euclidean length of roots.
\begin{lemma}\label{lm:SimpleRootRels}\label{c3.lm:SimpleRootRels}\label{c7.N_i}
 Let  $i\in M \in {\rm BCFG}$ be a node adjacent to another with double or triple bond.
Let $\beta\in \Psi^+$ be of the maximal height on the $W(M)$-orbit of $\beta_i$ and $r\in W(M)$ be the unique element of minimal length such that $\beta=r\beta_i$.
Let $e_{\beta}=re_ir^{-1}$. Then for each element $x$ in the stabilizer of $\beta$ in $W(M)$ (restricted to $\Psi^+$), we have
$$e_{\beta}=xe_{\beta}x^{-1}.$$
\end{lemma}
\begin{proof}
Consider it first for type $\ddC_n$ (\cite{CLY2010}). \\
There are two $W(\ddC_n)$-orbits with representatives $\beta_0$ and $\beta_1$, and  we suppose that $\beta_{-2}\in W(\ddC_n)\beta_0 $ and $\beta_{-1}\in W(\ddC_n)\beta_1$.\\
Let $i=1$ and $\beta=\beta_{-1}$.
  It is known that the extended Dynkin diagram of type $\ddC_n$ (just consider the group structure) arises as follows if we add $\beta_{-1}$ to the Dynkin diagram of type $\ddC_n$.
  $$\EBn$$
 According to \cite{BH1999}, the stabilizer of $\beta_{-1}$  in $W(\ddC_n)$ is generated by reflections in $W(\ddC_n)$ whose roots correspond to   nodes in the extended Dynkin diagram of type $\ddB_n$ which are nonadjacent to $-1$, as we just consider the action on positive roots,  the
stabilizer of $\{\beta_1\}$ is
$$N_{-1}=\left<r_0,r_1,\ldots, r_{n-3}, r_{n-1}, r_{-1} \right>,$$
where $r_{-1}=r_{\beta_{-1}}=rr_1r^{-1}$ and $r=r_{n-2}r_{n-1}r_{n-3}r_{n-2}\cdots r_2r_3r_1r_2r_0$.
Hence it suffices to prove the lemma holds for each generator of $N_{-1}$.
For $r_{-1}$,
\begin{eqnarray*}
r_{-1}e_{-1}&=&rr_1r^{-1}re_1r^{-1}=rr_1e_1r^{-1}\overset{(\ref{0.1.4})}{=}e_{-1},\\
e_{-1}r_{-1}&=&re_1r^{-1}rr_1r^{-1}=re_1r_1r^{-1}\overset{(\ref{0.1.4})}{=}e_{-1}.
\end{eqnarray*}
For $r_{n-1}$, we have
\begin{eqnarray*}
r_{n-1}e_{-1}r_{n-1}&=&(r_{n-1}r_{n-2}r_{n-1})r_{n-3}r_{n-2}\cdots e_1\cdots r_{n-2}r_{n-3}(r_{n-1}r_{n-2}r_{n-1})\\
&\overset{(\ref{0.1.10})}{=}&r_{n-2}r_{n-1}(r_{n-2}r_{n-3}r_{n-2}\cdots e_1\cdots r_{n-2}r_{n-3}r_{n-2})r_{n-1}r_{n-2},
\end{eqnarray*}
by induction, therefore it is reduced to prove that
\begin{eqnarray*}r_1r_0e_1r_0r_1=r_0e_1r_0,
\end{eqnarray*}
which holds for
\begin{eqnarray*}
(r_1r_0e_1)r_0r_1\overset{(\ref{0.1.14})}{=}r_0(e_1r_0r_1)\overset{(\ref{0.1.17})}{=}r_0e_1r_0.
\end{eqnarray*}
For $r_0$, we have
\begin{eqnarray*}
r_0e_{-1}r_0&=&r_0r_{n-2}r_{n-1}r_{n-3}r_{n-2}\cdots e_1\cdots r_{n-2}r_{n-3}r_{n-1}r_{n-2}r_0\\
&\overset{(\ref{0.1.7})}{=}& r_{n-2}r_{n-1}\cdots r_0r_1(r_2r_0)e_1(r_0r_2)r_1r_0\cdots r_{n-1}r_{n-2}\\
&\overset{(\ref{0.1.7})}{=}& r_{n-2}r_{n-1}\cdots r_0r_1r_0(r_2e_1r_2)r_0r_1r_0\cdots r_{n-1}r_{n-2}\\
&\overset{(\ref{0.1.15})}{=}& r_{n-2}r_{n-1}\cdots (r_0r_1r_0r_1)e_2(r_1r_0r_1r_0)\cdots r_{n-1}r_{n-2}\\
&\overset{(\ref{0.1.11})}{=}& r_{n-2}r_{n-1}\cdots r_1r_0r_1(r_0e_2r_0)r_1r_0r_1\cdots r_{n-1}r_{n-2}\\
&\overset{(\ref{0.1.8})+(\ref{0.1.3})}{=}& r_{n-2}r_{n-1}\cdots r_1r_0(r_1e_2r_1)r_0r_1\cdots r_{n-1}r_{n-2}\\
&\overset{(\ref{0.1.15})}{=}& r_{n-2}r_{n-1}\cdots r_1(r_0r_2)e_1(r_2r_0)r_1\cdots r_{n-1}r_{n-2}\\
&\overset{(\ref{0.1.7})}{=}& r_{n-2}r_{n-1}\cdots r_1r_2r_0e_1r_0r_2r_1\cdots r_{n-1}r_{n-2}=e_{-1}.
\end{eqnarray*}
For $r_{n-3}$, we have
\begin{eqnarray*}
r_{n-3}e_{-1}r_{n-3}&=&(r_{n-3}r_{n-2}r_{n-1}r_{n-3})r_{n-2}\cdots e_1\cdots r_{n-2}(r_{n-3}r_{n-1}r_{n-2}r_{n-3})\\
&\overset{(\ref{0.1.7})+(\ref{0.1.10})}{=}&r_{n-2}r_{n-3}(r_{n-2}r_{n-1}r_{n-2})\cdots e_1\cdots (r_{n-2}r_{n-1}r_{n-2})r_{n-3}r_{n-2}\\
&\overset{(\ref{0.1.10})}{=}&r_{n-2}r_{n-3}r_{n-1}r_{n-2}(r_{n-1}\cdots e_1\cdots r_{n-1})r_{n-2}r_{n-1}r_{n-3}r_{n-2}\\
&\overset{(\ref{0.1.7})+(\ref{0.1.8})}{=}&r_{n-2}r_{n-3}r_{n-1}r_{n-2}\cdots e_1\cdots (r_{n-1}r_{n-1})r_{n-2}r_{n-1}r_{n-3}r_{n-2}\\
&\overset{(\ref{0.1.3})}{=}&r_{n-2}(r_{n-3}r_{n-1})r_{n-2}\cdots e_1\cdots r_{n-2}(r_{n-1}r_{n-3})r_{n-2}\\
&\overset{(\ref{0.1.3})}{=}&r_{n-2}r_{n-1}r_{n-3}r_{n-2}\cdots e_1\cdots r_{n-2}r_{n-3}r_{n-1}r_{n-2}=e_{-1}.
\end{eqnarray*}
For any $r_i$ with $1\leq i\leq n-4$, we can prove the required equality  by induction as $r_{n-3}$. \\
Let $i=0$ and $\beta=\beta_{-2}$.
 It is known  that  the extended Dynkin diagram of type $\ddC_n$  arises if we add $\beta_{-2}$ to the Dynkin diagram of type $\ddC_n$.
$$\ECn$$
 Similarly, the stabilizer of $\{\beta_{-2}\}$ is
 $$N_{-2}=\left<r_0,r_1,\ldots, r_{n-3}, r_{n-2}, r_{-2} \right>,$$
 where $r_{-2}=r_{\beta_{-2}}=rr_1r^{-1}$ and $r=r_{n-1}r_{n-2}\cdots r_3r_2r_1$.\\
 Hence it also suffices to prove the lemma holds for each generator of $N_{-2}$. \\
It can be easily verified that $r_{-2}e_{-2}=e_{-2}r_{-2}=e_{-2}$ for $r_{-2}$.\\
 For $r_{n-2}$, we have
 \begin{eqnarray*}
 r_0e_{-2}r_0&=&(r_0r_{n-1}r_{n-2}\cdots r_2)r_1e_0r_1(r_2\cdots r_{n-2}r_{n-1}r_0)\\
             &\overset{(\ref{0.1.7})}{=}&r_{n-1}r_{n-2}\cdots r_2r_0r_1e_0r_1r_0r_2\cdots r_{n-2}r_{n-1}\\
             &\overset{(\ref{0.1.3})}{=}&r_{n-1}r_{n-2}\cdots r_2r_1(r_1r_0r_1e_0)r_1r_0r_1r_1r_2\cdots r_{n-2}r_{n-1}\\
             &\overset{(\ref{0.1.20})}{=}&r_{n-1}r_{n-2}\cdots r_2r_1e_0(r_1r_0r_1r_1r_0r_1)r_1r_2\cdots r_{n-2}r_{n-1}\\
             &\overset{(\ref{0.1.3})}{=}&r_{n-1}r_{n-2}\cdots r_2r_1e_0r_1r_2\cdots r_{n-2}r_{n-1}=e_{-2}.
 \end{eqnarray*}
 For $r_{n-2}$, we have
 \begin{eqnarray*}
 r_{n-2}e_{-2}r_{n-2}&=&(r_{n-2}r_{n-1}r_{n-2})\cdots r_2r_1e_0r_1r_2\cdots (r_{n-2}r_{n-1}r_{n-2})\\
 &\overset{(\ref{0.1.10})}{=}&r_{n-1}r_{n-2}(r_{n-1}\cdots r_2r_1e_0r_1r_2\cdots) r_{n-1}r_{n-2}r_{n-1}\\
 &\overset{(\ref{0.1.7})+(\ref{0.1.8})}{=}&r_{n-1}r_{n-2}\cdots r_3r_2r_1e_0r_1r_2\cdots (r_{n-1}r_{n-1})r_{n-2}r_{n-1}\\
 &\overset{(\ref{0.1.7})+(\ref{0.1.8})}{=}&r_{n-1}r_{n-2}\cdots r_3r_2r_1e_0r_1r_2\cdots r_{n-2}r_{n-1}=e_{-2}.
 \end{eqnarray*}
 Therefore we prove the lemma for $M=\ddC_n$.\\
 If $M=\ddB_n$, the argument is nearly the same as that  for $\ddC_n$, even the formulas for $e_{-1}$ and $e_{-2}$ are kept. The only difference is choosing  formulas from (\ref{0.1.11})--(\ref{0.1.18}) carefully for the alternative relations between $0$ and $1$. \\
 For $M=\ddF_4$ and $\ddG_2$, it is quite similar with the corresponding extended
 Dynkin diagrams changed as below. We will do this for type $\ddG_2$ in Section \ref{sectnormformG2}.
 $$\Fo$$
 $$\Ft$$
 $$\Go$$
 $$\Gf$$
  \end{proof}
 Because each element in the positive root system of type $M$ must be on the Weyl group $W(M)$-orbits of the some simple root,   then we can
 define
 \begin{eqnarray}\label{genebeta}
we_{\beta}w^{-1}=e_{w\beta},
\end{eqnarray}
for $w\in W(M)$ and $\beta$ a root
of $W(M)$ for $M \in {\rm BCFG}$, which is well defined due to the Lemma \ref{lm:SimpleRootRels}.
 Note that $e_{\beta}=e_{-\beta}$.

We continue by discussing some desirable properties of the Brauer algebra
$\Br(M)$ for $M\in {\rm BCFG}$. First of all, for any disjoint union $J$ of diagrams of  simply laced type
$Q$,  the Brauer algebra is defined as the direct product of the
Brauer algebras whose types are the components of $J$. The parabolic property for
$M\in {\rm ADE}$ has been proved in \cite{CW2011}, which says that the subdiagram relation
implies the subalgebra relation. The next result
states that parabolic subalgebras of $\Br(M)$ behave well with the hypothesis that the subalgebra relation in
Table \ref{mainresults}. Here we define a parabolic subalgebra $A$ of $\Br(M)$ is a subalgebra generated by reflections in
in $W(M)$ and the corresponding Temperley-Lieb elements, namely $A$ is generated by $r_{\beta}$
and $e_{\beta}$ for  $\beta\in S$, $S$ a subset of $\Psi^+$. For the classical theory in \cite{B2002}, it is known that  a parabolic
subgroup of Weyl group  is  conjugate to a subgroup genrated by subdiagram of the Dynkin diagram,   by (\ref{genebeta}), then it follows that
a parabolic subalgebra of $\Br(M)$  is conjugate to some parabolic subalgebra generated by the reflections and  the Temperley-Lieb elements
of some subdiagram of $M$.
\begin{prop}\label{prop:parabolic}
Let $J$ be a set of nodes of the Dynkin diagram $M\in {\rm BCFG}$.
Then the parabolic subalgebra of the Brauer algebra $\Br(M)$,
that is, the subalgebra generated by $\{r_j,e_j\}_{j\in J}$,
is isomorphic to the Brauer algebra of type $J$.
\end{prop}

\begin{proof}
We just do the example for $M=\ddC_n$, the remaining cases can be verified by similar arguments.
Here we apply the diagram representation of $\Br(\ddC_n)$ from \cite{CLY2010}.
 In fact, the algebra  $\Br(\ddB_n)$ has a diagram representation inherited from $\Br(\ddD_{n+1})$ (\cite{CL2011}). \\
In view of induction on $n-|J|$ and restriction to connected components of
$J$, it suffices to prove the result for $J = \{1,\ldots,n-1\}$ and for
$J=\{0,\ldots,n-2\}$. In the former case, the type is $\ddA_{n-1}$ and the
statement follows from the observation that the symmetric diagrams without
strands crossing the vertical line through the middle of the segments
connecting the dots $(n,1)$ and $(n+1,1)$ are equal in number to the Brauer
diagrams on the $2n$ nodes (realized to the left
of the vertical line). In the latter case, the type is $\ddC_{n-1}$ and the
statement follows from the observation that the symmetric diagrams with
vertical strands from $(1,1)$ to $(1,0)$ and from $(2n,1)$ to $(2n,0)$ are
equal in number to the symmetric diagrams related to $\BrM(\ddC_{n-1})$.
\end{proof}
In \cite[Theorem 1]{Muehl92}, the nontrivial diagram automorphisms on Dynkin diagrams were generalized as admissible partitions(not related to our
admissible set here). We introduce the definition and a theorem from \cite{Muehl92}. In this paper, we  apply the same idea to obtain the Brauer algebras of type $\ddI_2^n$,  $\ddH_3$, $\ddH_4$ from
Brauer algebras of type $\ddA_{n-1}$, $\ddD_6$, $\ddE_8$, respectively, which are done in \cite{L2014}, \cite{L20132}.

\begin{defn}\label{defn:addpart}
\index{\emph{admissible partition}}
 Let $(W,S)$ be a Coxeter system of type $M$. If $J$ is a spherical subset of $S$,
 we write $w_J$ for the longest element of $W_J$, the parabolic subgroup of
 $W$ generated by $J$. If   $\mathbb{P}$ is a partition of $S$ all of whose parts are spherical,
 \\ we denote $C_W( \mathbb{P})$ the subgroup of $W$ generated by all $w_J$ for $J\in \mathbb{P}$.
  \\ A partition $\mathbb{P}$ of $S$
is called an admissible partition if, for each part $A\in\mathbb{P}$, the Coxeter
diagram restricted to $A$ is spherical and, for each element $w\in C_W( \mathbb{P})$, either
$l(rw)<l(w)$ for all $r\in A$ or $l(rw)>l(w)$ for all $r\in A$.
\end{defn}
\begin{thm} Let $\mathbb{P}$ be an admissible partition of $S$ and write $S_{\mathbb{P}}=\{w_{A}\mid A\in \mathbb{P}\}$.
Then the pair $(C_W(\mathbb{P}), S_{\mathbb{P}})$ is a Coxeter system. Its type is the Coxeter diagram
$M_{\mathbb{P}}$ on $S_{\mathbb{P}}$ whose $A$, $B$-entry is the order of $w_Aw_B$ in W.
\end{thm}
\begin{rem}With the main theorems in \cite{L2014} and \cite{L20132},
if we write $\ddH_2=\ddI_2^5$ and have proved the conclusion in Table \ref{mainresults}, the
commutative diagram below follows.

\begin{center}
\phantom{longgggg}
\quad
\xymatrix{
\Br(\ddH_2)\ar[r]\ar[d] & \Br(\ddH_3) \ar[d]\ar[r] &\Br(\ddH_4)\ar[d]\\
          \Br(\ddA_4) \ar[r] & \Br(\ddD_6)\ar[r]&\Br(\ddE_8)
          }
\end{center}
Since the homomorphisms in the second row and the vertical homomorphisms are injective,
it follows that the homomorphisms in the first  rows are also injective.
Therefore the  Proposition \ref{prop:parabolic}  for  $\Br(\ddH_3)$ and  $\Br(\ddH_4)$ can be verified.
\end{rem}

\section{Main theorm on $\ddG_2$}
\label{sectionmaintheoremonG2}
We wrote about obtaining the Coxeter group of type $\ddG_2$ from the one of type $\ddD_4$ in Section \ref{defn:defnnonsimply}.
In the following sections, we  carry out the analogous operation on $\Br(\ddD_4)$ to obtain $\Br(\ddG_2)$ as a subalgebra of it. First recall the definition
of $\Br(\ddG_2)$ in Definition \ref{0.1}.
%
In Section \ref{defn:defnnonsimply}, we introduce that   $W(\ddG_2)$ can be obtained as a subgroup of
 $W(\ddD_4)$ as the fixed  subgroup of the isomorphism $\sigma=(1,2,4)$ acting on the
generators of $W(\ddD_4)$ indicated in Figure \ref{D4}.
\begin{figure}
\begin{center}
\includegraphics[width=.7\textwidth,height=.2\textheight]{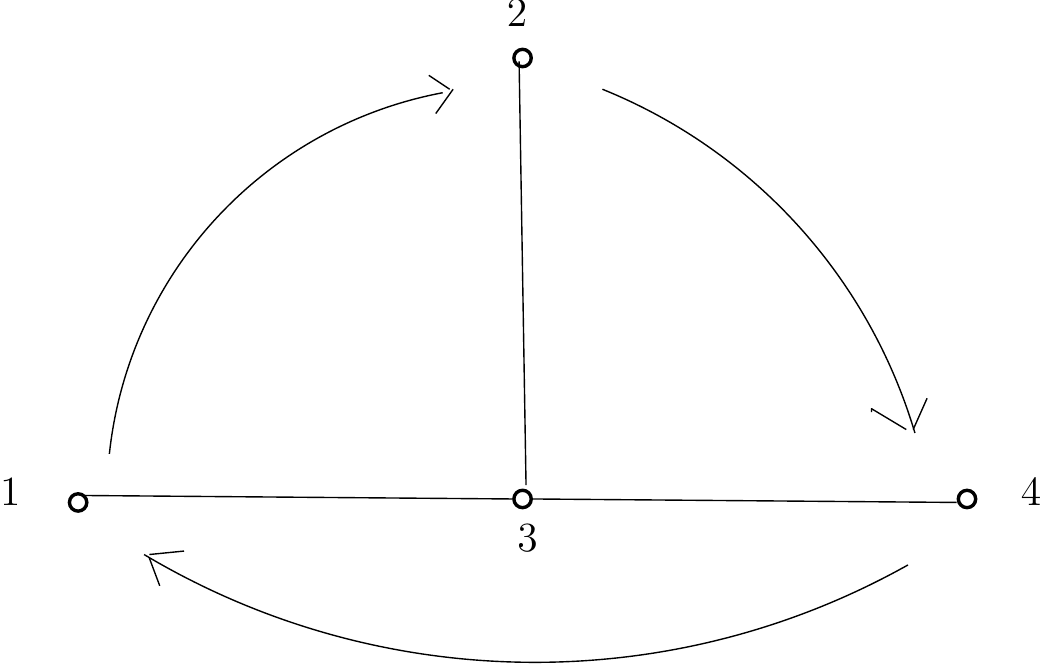}
\end{center}
\caption{Dynkin diagram of type $\ddD_4$}\label{D4}
\end{figure}

Here  the action of $\sigma$ can be extended to an isomorphism
 onto the Brauer algebra of type $\ddD_4$ by acting on the Temperley-Lieb generators $E_i$'s on their indices.
We denote  $\SBr(\ddD_4)$  the subalgebra of $\Br(\ddD_4)$ generated by  $\sigma$-invariant elements
in $\BrM(\ddD_4)$.
The main theorem on $\Br(\ddG_2)$ can be stated as  follows.
\begin{thm}\label{c7.main}
There is an  algebra isomorphism
$$\phi:\, \Br(\ddG_2)\longrightarrow \SBr(\ddD_4)$$
defined by $\phi(r_0)=R_1R_2R_4$, $\phi(r_1)=R_3$,
$\phi(e_0)=E_1E_2E_4$ and $\phi(e_1)=E_3$. Furthermore, both  $\Br(\ddG_2)$ and $\SBr(\ddD_4)$ are free over $\Z[\delta^{\pm 1}]$ of rank
$39$.
\end{thm}

\section{Root systems for $\ddD_4$ and $\ddG_2$}
Let $\{\alpha_i\}_{i=1}^{4}$ be simple roots of
$\ddD_4$. These can be realized in $\R^4$ with $\alpha_{1}=\eps_1+\eps_2$,
$\alpha_i=\eps_{i}-\eps_{i-1}$ for $2\leq i\leq 4$, where $\{\eps_i\}_{i=1}^{4}$ is an
orthonormal basis of $\R^4$. The positive root system will be taken to the vectors
$\eps_j\pm \eps_i$ with $1\leq i< j \leq 4$, denoted by $\Phi^+$, the whole root system
 is denoted by $\Phi$. We can define a  linear transformation by the following matrix, also called $\sigma$,
by considering each element in $\R^4$ as a column vector.
\begin{eqnarray*}
\left(
  \begin{array}{cccc}
    -1/2 & -1/2 & -1/2 & 1/2 \\
    1/2 & 1/2 & -1/2& 1/2 \\
    1/2 & -1/2 & 1/2 & 1/2 \\
    -1/2 & 1/2 & 1/2 & 1/2 \\
  \end{array}
\right)
\end{eqnarray*}
It can be checked that $\sigma(\alpha_1)=\alpha_2$,
$\sigma(\alpha_2)=\alpha_4$, $\sigma(\alpha_3)=\alpha_3$,
$\sigma(\alpha_4)=\alpha_1$ and $\sigma$ has order $3$ in ${\rm SL}(\R^4)$.
\\The Reynold's map is a surjective  linear map from $\R^4$ to the subspace of
 $\sigma$-invariant  vectors, defined as follows,
$$\R^4\longrightarrow \R_{\sigma}^4$$
$$\fp\,  : \,x \mapsto \frac{x+\sigma(x)+\sigma^2(x)}{3}.$$
In particular, $$\beta_0=\fp(\alpha_1)=\frac{\alpha_1+\alpha_2+\alpha_4}{3},\quad \,\beta_1=\fp(\alpha_3)=\alpha_3$$
is a basis of $\R_{\sigma}^4$ whose elements have square norm $2/3$ and $2$, respectively; these two vectors  can be taken as
simple roots of $\ddG_2$. The  root system of type $\ddG_2$ is denoted by
$\Psi$, and the positive root system  in which  $\beta_0$ and $\beta_1$ are simple roots is denoted by $\Psi^+$.

Recall the notion of  admissible set and $\cB$ from Definition \ref{df:admissible}.
There is a natural action of the Coxeter group $W(\ddG_2)$ ($W(\ddD_4)$, respectively)
acting on $\cB$ ($\cA$, respectively), by negating roots in
$\Psi\setminus \Psi^+$ ($\Phi\setminus \Phi^+$, respectively).
The following can be verified  by calculation.

\begin{prop} \label{c7.rem:ad}
The set  $\cB$ has  two $W(\ddG_2)$-orbits, which are
orbits of $\{\beta_1\}$ and $\{\beta_0, 3\beta_0 +2\beta_1\}$.
\end{prop}

\section{The map $\phi$ inducing a homomorphism}

In order to avoid confusion
with the above generators, the symbols of  the generators of $\Br(\ddD_4)$ have been capitalized.

\begin{rem} In Section \ref{sect:defnsimlacety}, we have seen that for each $\gamma\in \Phi$ and  each orthogonal root set $X\subset\Phi^+$, the monomials
 $E_{\gamma}$ and
$E_{X}$ are well defined. Let $R_\gamma$ denote the reflection corresponding to $\gamma$. By conjugation of elements of $W(\ddD_4)$,
 the symbol "$\sim$" can have a more general meaning of  two roots  being  not orthogonal or equal.
\end{rem}
We prove that $\phi$ induces a homomorphism.
\begin{proof}It suffices to  verify that the relations involved in Definition
\ref{c7.0.1} for type $\ddG_2$  still hold when the generators are  substituted by their
images under $\phi$.

As for (\ref{c7.0.1.7}), (\ref{c7.0.1.12})--(\ref{c7.0.1.14}), they follow from
 \begin{eqnarray*}
\phi(r_0)\phi(e_1)\phi(e_0)&=&R_1R_2(R_4E_3E_4)E_1E_2\\
&\overset{(\ref{3.1.2})}{=}&R_1(R_2R_3(E_4E_2))E_1\overset{(\ref{1.1.7})+(\ref{3.1.2})}{=}R_1E_3E_4E_2E_1\\
&\overset{(\ref{1.1.7})+(\ref{3.1.2})}{=}&R_3E_4E_2E_1=\phi(r_1)\phi(e_0),
\end{eqnarray*}
\begin{eqnarray*}
\phi(e_1)\phi(r_0)\phi(e_1)\phi(r_0)\phi(e_1)=E_3E_{\Sigma_{i=1}^{4}\alpha_i}E_3\overset{(\ref{3.1.5})}{=}E_3=\phi(e_1),
\end{eqnarray*}
\begin{eqnarray*}\phi(e_1)\phi(r_0)\phi(e_1)\phi(r_0)\phi(r_1)&=&E_3E_{\Sigma_{i=1}^{4}\alpha_i}R_3\\
&\overset{(\ref{3.1.3})}{=}&E_3R_{\Sigma_{i=1}^{4}\alpha_i}=\phi(e_1)\phi(r_0)\phi(r_1)\phi(r_0),
\end{eqnarray*}
\begin{eqnarray*}
\phi(e_0)\phi(r_1)\phi(e_0)&=&E_1E_2(E_4R_3E_4)E_1E_2\\
&\overset{(\ref{3.1.3})}{=}&E_1E_2E_4E_2E_1\\
&\overset{(\ref{1.1.7})+(\ref{1.1.4})}{=}&\delta^2 E_1E_2E_4=\phi(\delta^2 e_0).
\end{eqnarray*}
Relations (\ref{c7.0.1.8}) and (\ref{c7.0.1.16}) under $\phi$ acting on generators hold by the above verification for
(\ref{c7.0.1.7}), (\ref{c7.0.1.13}) and the natural opposition involution
on $\Br(\ddD_4)$.\\
The remaining relations can be proved similarly.
\end{proof}
\section{Normal forms of $\BrM(\ddG_2)$}\label{sectnormformG2}
The following lemma can be shown to hold  by an argument similar to the proof of
\cite[Lemma 4.1]{L2014}.
\begin{lemma} The submonoid of  $\BrM(\ddG_2)$  generated by $r_0$ and $r_1$ is
isomorphic to $W(\ddG_2)$.
\end{lemma}
\begin{lemma}\label{c7.2}
The following equalities hold in $\Br(\ddG_2)$.
\begin{eqnarray}
r_0r_1e_0&=&e_1e_0,     \label{c7.0.1.9}
\\
 e_0r_1r_0&=&e_0e_1,   \label{c7.0.1.10}
\\
 e_0e_1e_0&=&\delta^2 e_0,                                       \label{c7.0.1.11}
 \\
r_1r_0e_1r_0r_1e_0&=&\delta e_0.         \label{c7.0.1.15}
\end{eqnarray}
\end{lemma}
\begin{proof} These equalities are derived as follows.
\begin{eqnarray*}
r_0(r_1e_0)&\overset{(\ref{c7.0.1.7})}{=}&(r_0r_0)e_1e_0
\overset{(\ref{c7.0.1.3})}{=}e_1e_0,\\
e_0r_1r_0&\overset{(\ref{c7.0.1.8})}{=}&e_0e_1r_0r_0
\overset{(\ref{c7.0.1.3})}{=}e_0e_1,\\
e_0e_1e_0&\overset{(\ref{c7.0.1.9})}{=}&e_0r_0r_1e_0
\overset{(\ref{c7.0.1.4})}{=}e_0r_1e_0
\overset{(\ref{c7.0.1.14})}{=}\delta^2 e_0,\\
r_1r_0e_1(r_0r_1e_0)&\overset{(\ref{c7.0.1.9})}{=}&r_1r_0e_1 e_1e_0
\overset{(\ref{c7.0.1.6})}{=}\delta r_1r_0e_1e_0
\overset{(\ref{c7.0.1.7})}{=}\delta r_1r_1e_0
\overset{(\ref{c7.0.1.3})}{=}\delta e_0.
\end{eqnarray*}
\end{proof}
\begin{rem}Consider (\ref{c7.0.1.15}).
The image of the left side under $\phi$ is $E_{X^{\rm cl}}$, and of
the right side under $\phi$ is $\delta E_X$,
where $$X=\{\alpha_1, \alpha_2, \alpha_4\}, \quad X^{\rm cl}=\{\alpha_1, \alpha_2, \alpha_4, 2\alpha_3+\alpha_1+\alpha_2+\alpha_4\}.$$
It is known  from  Section \ref{sect:defnsimlacety} that $E_{X^{\rm cl}}=\delta E_{X}$,
this also serves for (\ref{c7.0.1.15}) holding under $\phi$.
\end{rem}

%
For the group $W(\ddG_2)$ acting on $\Psi^+$, we have the following conclusion.
\begin{lemma}\label{N_i}
 Let $N_0$, $N_1$ be  stabilizers of $\beta_0$ and $\beta_1$ respectively.
Then for any element $a\in N_i$, we have that $ae_ia^{-1}=e_i$, for $i=0$, $1$.
\end{lemma}
\begin{proof}
It can be checked that $N_0\cong (\Z/2\Z)^2$, with generators $r_0$ and $r_1r_0r_1r_0r_1$.
Also you can get that $N_1\cong (\Z/2\Z)^2$, with generators $r_1$ and $r_0r_1r_0r_1r_0$.
Hence the Lemma hold for the below,
\begin{eqnarray*}r_0e_0r_0&\overset{(\ref{c7.0.1.4})}{=}& e_0,\\
r_1e_1r_1&\overset{(\ref{c7.0.1.4})}{=}& e_1,\\
r_1r_0r_1(r_0r_1e_0r_1r_0)r_1r_0r_1&\overset{(\ref{c7.0.1.9})+(\ref{c7.0.1.10})}{=}&r_1r_0(r_1e_1e_0e_1r_1)r_0r_1\\
&\overset{(\ref{c7.0.1.4})}{=}&r_1(r_0e_1e_0e_1r_0)r_1\\
&\overset{(\ref{c7.0.1.7})+(\ref{0.1.8})}{=}&r_1r_1e_0r_1r_1\\
&\overset{(\ref{c7.0.1.3})}{=}&e_0,\\
r_0r_1(r_0r_1r_0e_1r_0r_1r_0)r_1r_0&\overset{(\ref{c7.0.1.16})+(\ref{c7.0.1.13})}{=}&(r_0r_1r_1r_0)e_1r_0e_1r_0e_1(r_0r_1r_1r_0)\\
&\overset{(\ref{c7.0.1.3})}{=}&e_1r_0e_1r_0e_1\\
&\overset{(\ref{c7.0.1.12})}{=}&e_1.
\end{eqnarray*}
\end{proof}

Recall  the defining formula (\ref{genebeta}). Consider a positive root $\beta$ and a node $i$ of type $\ddG_2$. If there
exists $w\in W$ such that $w\beta_i=\beta$, then we can define the element
$e_{\beta}$ in $\BrM(\ddG_2)$ by
$$e_{\beta}=we_iw^{-1}.$$
The above lemma implies that $e_\beta$ is well defined.

\begin{lemma}\label{c7.anyr}
Let $D_i$ be a set of left coset representatives for $N_i$ in $W(\ddG_2)$ for $i=0$, $1$,
and $K_0=\langle 1\rangle\subset N_0$, $K_1=\langle r_0r_1r_0r_1r_0\rangle\subset N_1$.
Then for any $r\in W(\ddG_2)$, there exist $a\in D_i$ and  $b\in K_i$
such that
 $$re_i=ae_ib.$$
\end{lemma}
\begin{proof}From the proof of Lemma \ref{c7.N_i}, by Lemma \ref{c7.2} and Proposition \ref{prop:opp}, we see that
\begin{eqnarray*}r_1r_0r_1r_0r_1e_0&=&e_0,\\
r_0r_1r_0r_1r_0e_1&=&e_1r_0r_1r_0r_1r_0.
\end{eqnarray*}
Therefore our lemma holds by the analogous argument in Proposition \cite[Lemma 4.8]{CL2011},
\cite[Lemma 4.2]{L2014}, and \cite[Lemma 6.5]{CLY2010} through writing an element as the product of one element in $D_i$ and another in $N_i$.
\end{proof}

\begin{lemma}\label{c7.anye}
For any $\beta\in \Psi^+$ and  $i\in \{0, 1 \}$,
there exist $a_1$ $a_2\in W(\ddG_2)$, $j\in \{0,1 \}$ and $t\in \Z$ such that
$$e_\beta e_i=\delta^{t}a_1e_ja_2.$$
\end{lemma}
\begin{proof} It is known that
$$\Psi^+=\{\alpha_0, r_1\alpha_0, r_0r_1\alpha_0, \alpha_1, r_0\alpha_1, r_1r_0\alpha_1 \}.$$

We verify the lemma by distinguishing all possible cases.
\begin{eqnarray*}
e_0e_0&=&\delta^3 e_0,\\
r_1e_0r_1e_0&\overset{(\ref{c7.0.1.14})}{=}&\delta r_1e_0,\\
r_0r_1e_0r_1r_0e_0&\overset{(\ref{c7.0.1.4})+(\ref{c7.0.1.14})}{=}&\delta r_0r_1e_0,\\
e_1e_0&\overset{(\ref{c7.0.1.9})}{=}& r_0r_1e_0,\\
r_0e_1r_0e_0&\overset{(\ref{c7.0.1.9})+(\ref{c7.0.1.4})}{=}& r_0r_1e_0,\\
r_1r_0(e_1r_0r_1e_0)&\overset{(\ref{c7.0.1.15})}{=}& \delta e_0,\\
e_0e_1&\overset{(\ref{c7.0.1.10})}{=}& e_0r_1r_0,\\
r_1e_0r_1e_1&\overset{(\ref{c7.0.1.10})+(\ref{c7.0.1.4})}{=}& r_1 e_0r_1r_0,\\
r_0r_1(e_0r_1r_0e_1)&\overset{(\ref{c7.0.1.10})+(\ref{c7.0.1.6})}{=}&\delta r_0r_1e_0r_1r_0,\\
e_1e_1&=&\delta e_1,\\
r_0e_1r_0e_1&\overset{(\ref{c7.0.1.16})}{=}& r_1r_0r_1r_0 e_1,\\
r_1r_0e_1r_0r_1e_1&\overset{(\ref{c7.0.1.16})+(\ref{c7.0.1.4})}{=}&r_0r_1r_0 e_1.
\end{eqnarray*}
\end{proof}

\begin{lemma}\label{c7.rewrittenforms}
 Each element in $\BrM(\ddG_2)$ can be written in one of the following normal forms
\begin{enumerate}[(i)]
 \item $\delta^k u e_i v w$  or
\item  $\delta^k a$
\end{enumerate}
where $i=0$ or $1$,  $u\in D_i$, $v\in K_i$ (Lemma \ref{c7.anyr}), $w\in D_i^{\rm op}$, $k\in \Z$, $a\in W(\ddG_2)$.
In particular,  $\Br(\ddG_2)$ is spanned by $39$ elements over $\Z[\delta^{\pm 1}]$.
\end{lemma}
\begin{proof}It suffices to prove that  the set of the normal forms is closed under
 multiplication.
By Proposition \ref{c7.prop:opp} and Lemma \ref{c7.anyr}, the product of two elements of the normal forms  are still in the set of normal forms when one of the  two elements
belongs to the second forms.
 If two elements $\delta^{k_1} u_1 e_i v_1 w_1$, $\delta^{k_2} u_2 e_j v_2 w_2$
 are in the  first forms, we have
 \begin{eqnarray*}
 &&\delta^{k_1} u_1 e_i v_1 w_1\delta^{k_2} u_2 e_j v_2 w_2\\
 &=&\delta^{k_1+k_2}u_1v_1 w_1 u_2((v_1 w_1 u_2)^{-1}  e_i v_1 w_1 u_2) e_j v_2 w_2\\
 &\in&\delta^{k} W(\ddG_2)e_{t}W(\ddG_2) \quad \quad \mbox{Lemma \ref{c7.anye}}.
 \end{eqnarray*}
 Next we can apply  Proposition
\ref{c7.prop:opp} and Lemma \ref{c7.anyr} to conclude  that the product is a normal form.\\
The last claim follows from
$$\#(W(\ddG_2))+\#(D_0)^2\#(K_0)+\#(D_1)^2\#(K_1)=12+3^2\cdot 1 +3^2\cdot 2=39.$$
\end{proof}
\section{The algebra $\SBr(\ddD_4)$}
\label{sectiionSBrddD4}
Now  finish the proof of Theorem \ref{c7.main} by proving the surjectivity  of $\phi$.
\begin{proof}First  analogous to  \cite[Section 5]{L20132}, it can be checked that the image $\phi(\Br(\ddG_2))$ is free of rank $39$ by the normal forms in
Lemma \ref{c7.rewrittenforms} and  by the normal forms of  $\Br(\ddD_4)$ in Theorem \ref{thm:genralwriting} as follow.
\begin{enumerate}[(I)]
\item The  restriction of $\phi$ to normal forms in ${\rm (ii)}$ of  Lemma \ref{c7.rewrittenforms} is injective thanks to the embedding of $W(\ddG_2)$ to $W(\ddD_4)$.
\item Since there are three elements in   the $\phi(W(\ddG_2))$-orbit  of $\{\alpha_1,\alpha_2,\alpha_4\}^{\rm cl}$;
then the $\phi$ on normal forms in ${\rm (i)}$ of  Lemma \ref{c7.rewrittenforms}  with $i=0$ is an injective map to the normal forms in Theorem \ref{thm:genralwriting} associated to $Y=\{1,2,4\}$.
\item Since there are three elements in   the $\phi(W(\ddG_2))$-orbit  of $\{\alpha_3\}$ and we have
$1\neq \hat{E}_3\phi(r_0r_1r_0r_1r_0)\hat{E}_3\in W(M_{\{3\}})$ (by the diagram representation in \cite{CGW2008});
then
the $\phi$ on normal forms in ${\rm (i)}$ of  Lemma \ref{c7.rewrittenforms}  with $i=1$ is an injective map to the normal forms in Theorem \ref{thm:genralwriting} associated to $Y=\{3\}$.
\end{enumerate}
We know the admissible root  sets $Y$ of type $\ddD_4$ involved for the normal forms in Theorem \ref{thm:genralwriting}  are
$$\emptyset, \{3\},\{1,2\},\{1,4\},\{1,2,4\}.$$
If $a\in \SBrM(\ddD_4)$, we see that the monoid actions  of $a$ and $a^{\op}$ on $\emptyset$ are $\sigma$-invariant.
Because there is no $\sigma$-invariant
element in the $W(\ddD_4)$-orbits of $\{\alpha_1,\alpha_2\}$ and $\{\alpha_1,\alpha_4\}$,
the only possible $Y$ for normal forms of elements in  $ \SBrM(\ddD_4)$ in Theorem \ref{thm:genralwriting} are
$$\emptyset, \{3\},\{1,2,4\}.$$
Therefore applying Theorem \ref{thm:genralwriting},  we can verify that $\SBr(\ddD_4)$ is exactly equal to $\phi(\Br(\ddG_2))$ described  at the beginning of the proof and of rank $39$ as in \cite[Section 5]{L2013}. Namely, by Theorem \ref{thm:genralwriting}(\cite[Theorem 2.7]{CW2011}),
 we write each monoid in the sandwich form $\delta^{i} a_{B} \hat{E}_Y h a_{B'}^{\rm op}$, because we select the $Y(t)$ $\sigma$-invariant for $t=0$, $1$, $3$, and we just need to find those sandwich forms with $a_{B}$,  $h$, $a_{B'}^{\rm op}$ which are  $\sigma$-invariant, which implies that $B$, $h$, $B'$  are $\sigma$-invariant, and we see those $\sigma$-invariant sandwich forms are exactly the $\SBrM(\ddD_4)$ and the image of $\phi(\BrM(\ddG_2))$. This accomplishes the proof of Theorem \ref{c7.main}.
\end{proof}
Now we have the decomposition for $\Br(\ddG_2)$  as $\Z[\delta^{\pm 1}]$-module as follows,
\begin{eqnarray*}
\Br(\ddG_2)=\Br(\ddG_2)/(r_1r_0e_1r_0r_1)\oplus(r_1r_0e_1r_0r_1)/(e_0r_1r_0e_1r_0r_1)\oplus(e_0r_1r_0e_1r_0r_1).
\end{eqnarray*}
Similar to the arguments in  \cite{BO2011} or \cite[Section 6]{L2013}, with the conclusion about cellularity of Hecke algebra of
type $\ddG_2$ in \cite{G2007}, the  theorem below  about the cellularity  follows.
\begin{thm} If $R$ is a field with characteristic not equal to $2$ or $3$, then
$\Br(\ddG_2)\otimes R$  is a cellularly stratified algebra.
\end{thm}
\begin{rem}\label{rem:compaIG}
From \cite{L2014}, we know that $\Br(\ddI_2^6)$ is of rank $2\cdot 6 +3/2\cdot 6^2=66$.  We  see that
$\{\beta_0\}$ is an admissible root set of type $\ddI_2^6$, which is not true for type $\ddG_2$. This is   the reason for the rank difference between  $\Br(\ddI_2^6)$ and $\Br(\ddG_2)$. If $\delta=1$, there exists a surjective homomorphism $\varphi:\Br(\ddI_2^6)\rightarrow\Br(\ddG_2)$ determined by $\varphi(r_i)=r_{1-i}$ and $\varphi(e_i)=e_{1-i}$, for $i=0$, $1$.
 \end{rem}
{\bf Acknowledgement} At the end of the paper, I appreciate Prof A.~Cohen for his supervision for my PhD thesis on Brauer algebras of
non-simply laced type, and say thanks to  Prof E.~Opdam for his support in my postdoc research. Most of the paper is included in the author's thesis
(\cite{Lthesis}), and the author now is supported by the Fundamental Research Funds of Shandong University.

Shoumin Liu\\
Email: s.liu@sdu.edu.cn\\
Taishan College, Shandong University\\
Shanda Nanlu 27, Jinan, \\
Shandong Province, China\\
Postcode: 250100

%

\begin{thebibliography}{9}
\bibitem{B2002}N.~Bourbaki, Lie groups and Lie algebras, Chapter 4--6, Elements of mathematics,
2002 Springer-Verlag Berlin Heidelburg.
\bibitem{BO2011}
C. Bowman,
Brauer algebras of type $C$ are cellularly stratified algebras,
Mathematical Proceedings of the Cambridge Philosophical Society,
{\bf 153} 2012, 1--7.


\bibitem{Brauer1937}
R.~Brauer,
On algebras which are connected with the semisimple continuous groups,
Annals of Mathematics, {\bf 38} (1937), 857--872.
\bibitem{BH1999}
B. Brink, R. Howlett, Normalizers of parabolic subgroups in Coxeter groups,
Invent. math. {\bf 136} (1999), 323--351.
\bibitem{BC2011}
Francis.~Buekenhout, A.M.~Cohen, Diagram geometry,  A series of modern survey in mathematics
({\bf 57}), 2013, Springer.

\bibitem{Car}
R. Carter,
Simple groups of Lie type, Wiley classics library.

\bibitem{ZhiChen}
Zhi Chen,
Flat connections and Brauer type algebras,
Journal of algebra, {\bf 365}(2012),114--146.
\bibitem{CFW2008}A.M.~Cohen, B.~Frenk and D.B.~Wales, Brauer algebras
of simply laced type, Israel Journal of Mathematics, {\bf 173} (2009),
335--365.

\bibitem{CGW2005}A.M.~Cohen, Di\'{e} A.H.~Gijsbers and D.B.~Wales,
BMW algebras of simply laced types, J.Algebra, {\bf 286} (2005), 107--153.

\bibitem{CGW2006}A.M.~Cohen, Di\'{e} A.H.~Gijsbers and D.B.~Wales,
A poset connected to Artin monoids of simply laced type,
Journal of Combinatorial Theory, Series A {\bf 113} (2006), 1646--1666.

\bibitem{CGW2008}A.M.~Cohen, Di\'{e} A.H.~Gijsbers and D.B.~Wales,
The BMW algebras of type $\ddD_{n}$,
to appear in Communications in Algebra.

\bibitem{CGW2009}A.M.~Cohen, Di\'{e} A.H.~Gijsbers and D.B.~Wales,
Tangle and Brauer diagram algebras of type $\ddD_{n}$, Journal of  Knot Theory and its Ramifications,
 {\bf 18} (2009), 447--483.

\bibitem{CLY2010}A.M.~Cohen, S.~Liu and S.~Yu, Brauer algebras of type C,
Journal of Pure and Applied Algebra, {\bf 216} (2012), 407--426.

\bibitem{CL2011}A.M.~Cohen, S.~Liu, Brauer algebras of type B,
Forum Mathematicum, DOI: 10.1515/forum-2012-0041, March 2013.

\bibitem{CW2011} A.M.~Cohen and D.B.~Wales, The Birman-Murakami-Wenzl algebras of type $\ddE_n$,
Transformation Groups,  {\bf 16} (2011), 681--715.



%
%
%
%
%

\bibitem{F1995}K.~Fan,
A Hecke algebra quotient and properties of commutative elements of a Weyl group. PhD thesis,
MIT, May 1995.
\bibitem{G2007}
M. Geck, Hecke algebras of finite type are cellular, Invent. math.,
{\bf 169} (2007), 501--517.


\bibitem{Gra}
J.~J.~Graham, Modular representations
of Hecke algebras and related algebras, Ph.~D.~thesis, University of
Sydney (1995).


\bibitem{Lthesis}S.~Liu, Brauer algebras of non-simply laced type, PhD thesis, Technische Universiteit Eindhoven, 2012.
\bibitem{L2013}S.~Liu, Brauer algebra of type $\ddF_4$,
Indagationes Mathematicae, {\bf 24} (2013), 428–-442.
\bibitem{L2014}S.~Liu, Brauer algebra of type $\ddI_2^{n}$,
Journal of  Algebraic  Combinatorics, {\bf 40} (2014), 647--662.
\bibitem{L20132}S.~Liu, Brauer algebras of type $\ddH_3$
and $\ddH_4$, \url{ arXiv:1305.6528}, May 2013.
\bibitem{Muehl92}
{B. M\"uhlherr}, {Coxeter groups in Coxeter groups,}
pp.~277--287 in {Finite Geometry and Combinatorics (Deinze 1992).} London
Math.~Soc.~Lecture Note Series {\bf 191}, Cambridge University Press,
Cambridge, 1993.
\bibitem{S1971}
I.~Satake, Classification theory of semisimple algebraic groups. Lecture Notes in Pure and Appl. Math.,
Marcel Dekker, New York 1971.
\bibitem{S1998}
T.~A.~Springer, Linear algebraic groups, 2nd version, Birkh\"auser Boston, 1998.
\bibitem{TL1971} H.N.V. Temperley and E. Lieb, Relation between
percolation and colouring problems and other graph theoretical problems
associated with regular planar lattices: some exact results for the
percolation problems, Proc.~Royal.~Soc.~A. {\bf 322} (1971), 251--288.
\bibitem{T1959}
J.Tits,
Groupes alg\'ebriques semi-simples et g\'eom\'etries associ\'ees. In Algebraic
and topological Foundations of Geometry (Proc. Colloq., Utrecht, 1959).
Pergamon, Oxford, 175--192.

\end{thebibliography}
\end{document}